\documentclass[9pt, reqno, twoside]{amsart}

\usepackage[%
	thmreset=section,
	eqreset=section,
	alglinenumber,
]{myPreamble}
\ifams
	\usepackage[foot]{amsaddr}
\fi
\ifsiam
\fi
\let\oldthanks\thanks
\renewcommand{\thanks}[1]{\texorpdfstring{\oldthanks{#1}}{}}

	\allowdisplaybreaks

	\makeatletter
		\renewcommand{\clevethm@proofsectiontitle}{Omitted proofs of }
	\makeatother
	\crefname{ALG@line}{step}{steps}
	\crefname{enumeratpropi}{property}{properties}
	\crefname{enumeratpropii}{property}{properties}
	\counterwithout{algorithm}{section}
	\renewcommand{\thealgorithm}{\arabic{algorithm}}
	\crefname{section}{Section}{Sections}
	\crefname{subsection}{Subsection}{Subsections}

\DeclareMathAlphabet{\mathcalligra}{T1}{calligra}{m}{n}
	
\makeatletter
	\def\operator@font{\rm}

	\DeclareMathOperator{\BregmanD}{D}

	\newcommand{\BregmanKernel}{H}						
	\newcommand{\D}{\@ifstar\@@D\@D}					
	\newcommand{\h}{\@ifstar{\@@h}{\@h}}				
	\let\P\relax
	\newcommand{\E}{\@ifstar\@@E\@E}					
	\newcommand{\P}{\@ifstar\@@P\@P}					

        \newcommand{\@h}{\@ifnextchar_{h}{\BregmanKernel}}
    \newcommand{\@@h}{\@ifnextchar_{\hat h}{\hat\BregmanKernel}}

	\newcommand{\@D}{\@ifnextchar_{\@D@sub}{\BregmanD_{\@h}}}
	\newcommand{\@@D}{\@ifnextchar_{\@@D@sub}{\@@Dnosub}}
	\def\@D@sub_#1{\BregmanD_{\@h_{#1}}}
	\def\@@D@sub_#1{\BregmanD_{\@@h_{#1}}}

	\newcommand{\@@Dnosub}{\@ifstar{\@@D@@}{\BregmanD_{\@@h}}}
	\newcommand{\@@D@@}{\@ifnextchar_{\@@D@@sub@}{\BregmanD_{\@@h@}}}
    \newcommand{\@@h@}{\@ifnextchar_{\conj{\hat h}}{\conj{\hat\BregmanKernel}}}
	\def\@@D@@sub@_#1{\BregmanD_{\@@h@_{#1}}}

	\newcommand{\@P}[2][k]{\mathcal P_{#1}\ifstrempty{#2}{}{\left[#2\right]}}
	\newcommand{\@@P}[2][k]{\mathcal P_{#1}\ifstrempty{#2}{}{[#2]}}

	\newcommandx{\@E}[3][1=k,3={}]{
		\mathbb E_{#1}\ifstrempty{#2}{}{
			\left[
				#2\ifstrempty{#3}{}{\vphantom{#3}\right|\left.#3}
			\right]
		}
	}
	\newcommandx{\@@E}[3][1=k,3={}]{
		\mathbb E_{#1}\ifstrempty{#2}{}{
			[#2\ifstrempty{#3}{}{|#3}]
		}
	}

\makeatother

	\pgfplotsset{compat=1.16}

\showgrayfalse
 
	\newcommand{\TheTitle}{Tight
Convergence Rates of the Gradient Method on Smooth
Hypoconvex Functions}
	\newcommand{\TheShortTitle}{Convergence rate of the gradient method on smooth hypoconvex functions}
	
	\newcommand{\TheFunding}{%
		This research was supported by a grant from the Global PhD partnership between KU Leuven and UCLouvain, and the research projects G081222N, G0A0920N, G086518N, and G086318N; Research Council KU Leuven C1 project No. C14/18/068; Fonds de la Recherche Scientifique -- FNRS and the Fonds Wetenschappelijk Onderzoek -- Vlaanderen under EOS project no 30468160 (SeLMA); European Union's Horizon 2020 research and innovation programme under the Marie Skłodowska-Curie grant agreement No. 953348.
	}
	\newcommand{\TheKeywords}{%
		Performance estimation, 
		Gradient method, 
		Hypoconvex functions, 
		Convergence rates%
	}  
	
	\newcommand{\TheAbstract}{
	    We perform the first tight convergence analysis of the gradient method with varying step sizes when applied to smooth hypoconvex (weakly convex) functions. Hypoconvex functions are smooth nonconvex functions whose curvature is bounded and assumed to belong to the interval $[\mu, L]$, with $\mu<0$. Our convergence rates improve and extend the existing analysis for smooth nonconvex functions with $L$-Lipschitz gradient (which corresponds to the case $\mu=-L$), and smoothly interpolates between that class and the class of smooth convex functions. We obtain our results using the performance estimation framework adapted to hypoconvex functions, for which new interpolation conditions are derived. We derive explicit upper bounds on the minimum gradient norm of the iterates for a large range of step sizes, explain why all such rates share a common structure, and prove that these rates are tight when step sizes are smaller or equal to $1/L$. Finally, we identify the optimal constant step size that minimizes the worst-case of the gradient method applied to hypoconvex functions.
	}

	\title[\TheShortTitle]{\LARGE\normalfont\scshape\TheTitle}
	\author[T. Rotaru]{Teodor Rotaru\textsuperscript{1,2}}
		\address{\textsuperscript{1}\textnormal\TheAddressKU}
		\email{teodor.rotaru@kuleuven.be}
	\author[F. Glineur]{Fran\c cois Glineur \textsuperscript{2}}
	    \address{\textsuperscript{2}\textnormal\TheAddressUCLN}
		\email{francois.glineur@uclouvain.be}
	\author[P. Patrinos]{Panagiotis Patrinos\textsuperscript{1}}
		\email{panos.patrinos@esat.kuleuven.be}
	
	\thanks{\TheFunding}


\graphicspath{{./pics/}}

\theoremstyle{plain}
\newtheorem{theorem}{Theorem}[section]
\newtheorem{proposition}[theorem]{Proposition}
\newtheorem{lemma}[theorem]{Lemma}

\theoremstyle{definition}
\newtheorem{definition}[theorem]{Definition}

\newtheorem{conjecture}[theorem]{Conjecture}
\theoremstyle{remark}

\begin{document}
	
	\ifsiam
		\maketitle
	\fi
	\begin{abstract}\TheAbstract\end{abstract}
	
	\ifams
		\maketitle
	\else
	\fi

	\section{Introduction}\label{sec:Introduction}
%

The problem of identifying tight convergence rates for optimization methods on different classes of functions is receiving increasing attention. While most of the existing work targets convex functions, non convex functions are very frequently encountered in the large-scale optimization problems that need to be solved in machine learning. 
In this work we provide such insights about the class of smooth hypoconvex functions, a specific subset of nonconvex functions. 
We consider the canonical first-order optimization algorithm, i.e., the gradient method with varying step sizes, and derive tight convergence results for a large range of step sizes.


\subsection{Tight convergence analysis and performance estimation}\label{subsec:Tight_conv_literature}
When solving an optimization problem through an iterative method, one is often interested in the evolution of some convergence measure after applying a given number of steps. Drori and Teboulle tackle this problem in an novel way in \cite{drori_performance_2014}: they introduce a tool for the worst-case analysis of first-order methods -- the Performance Estimation Problem (PEP). A PEP consists in modeling the problem of finding the \textit{worst} behavior of a given algorithm on a given problem class as an optimization problem. In many cases this problem can be cast as a convex semidefinite programming problem, hence is efficiently solvable.

Intuitively, some convergence measure that characterizes how far the iterates are from the solution after performing $N$ steps of the optimization method $\mathcal{M}$ is maximized with respect to a class of functions $\mathcal{F}$. A condition on the initial iterate that ensure boundedness of the problem is usually also required, such as a bound on the distance to a solution, $\|x_0 - x_*\|$. 
A general form of the PEP problem is the following optimization program:
{\normalsize
\begin{align}\label{eq:PEP_general_abstract}
    \begin{aligned}
        \maximize_{f\,,\, x_0} \quad & \mathcal{P}_{\mathcal{M},N} (f,x_0) \\
        \stt \quad & f \in \mathcal{F}  \\
        &\text{Initial conditions on } x_0 
    \end{aligned}
\end{align}
}%
where $\mathcal{P}_{\mathcal{M},N}(f,x_0)$ is the performance measure of method $\mathcal{M}$ after performing $N$ steps on function $f$ starting from $x_0$. Such performance measure can be for example the accuracy of the last iterate $f(x_N)-f(x_*)$. 
Decision variables are the function $f$ belonging to  a generic class of functions $\mathcal{F}$, and the initial iterate $x_0$.

Problem \eqref{eq:PEP_general_abstract} is infinite dimensional because the maximization is over a class of functions. Therefore, it has to be reformulated as a finite dimensional problem. A rigorous discretization is done by using necessary and sufficient interpolating conditions of functions from $\mathcal{F}$. These conditions were first established in the context of PEP by Taylor et. al., in \cite{taylor_smooth_2017}.

Worst-case convergence rates of first-order methods applied to smooth convex functions have been extensively analyzed with PEP. For instance, the gradient method applied to convex functions with step sizes shorter than $\tfrac{1}{L}$, where $L$ is the Lipschitz constant of the gradient, was studied in \cite{drori_performance_2014}, while the strongly-convex functions were analyzed in \cite{taylor_smooth_2017} for fixed step sizes and in \cite{deKlerk_Taylor_line_search_2017} when exact line-search is performed. Optimized first-order methods can also be obtained within the PEP framework, see for example \cite{Kim2021_PEP_Cvx} for an optimal algorithm to decrease the gradient norm of smooth convex functions.

\subsection{Previous analysis of the gradient method on smooth nonconvex functions}\label{subsec:Nonconvex_literature}
Exact convergence rates for the larger class of smooth nonconvex functions have only started to be studied relatively recently. 
First, as one cannot guarantee the convergence of first-order algorithms to a global minimizer, we note that such convergence analysis is made with respect to finding a stationary point. 
Hence the performance of algorithms for smooth nonconvex functions is measured in terms of the minimum gradient norm among all iterates (the minimum being needed since the gradient norm may not be monotonically decreasing).


Within the PEP framework for tight analysis, the gradient method for smooth  nonconvex functions was studied by Taylor in \cite[page 190]{PhD_AT_2017} for the particular constant step size equal to $\tfrac{1}{L}$. The result was extended by Drori and Shamir in \cite[Corollary 1]{drori2021complexity} for step sizes shorter than $\tfrac{1}{L}$ and then improved by Abbaszadehpeivasti et. al., in \cite{abbaszadehpeivasti2021GM_smooth} for larger step sizes up to $\tfrac{\sqrt{3}}{L}$. All the above results were obtained for the class of smooth nonconvex functions, i.e., smooth functions whose gradient is assumed to be $L$-Lipschitz or, equivalently, whose curvature is assumed to belong to the interval $[-L, L]$ (recall that when a function is $\mathcal{C}^2$, its curvature lies between the minimum and maximum eigenvalues of the Hessian). 

In this work, we generalize this setting and consider smooth nonconvex functions with bounded curvature. More precisely we assume curvature belongs to the $[\mu,L]$ interval with  $\mu<0$, and call these functions hypoconvex. Our convergence rates improve on the existing analysis for smooth nonconvex functions with $L$-Lipschitz gradient when $-L<\mu<0$ (the extra information about the lower curvature allowing tighter bounds than on $[-L,L]$), and extend it when $\mu<-L$. They smoothly interpolate between the classes of smooth nonconvex functions and smooth convex functions with $L$-Lipschitz gradients. 


\subsection{Contributions}\label{subsec:Contributions}
We prove the following results using the technique of performance estimation, relying on Theorem \ref{thm:interp_hypo_characterization_min} which introduces the interpolation conditions for smooth hypoconvex functions, together with a characterization of their global minimum:
:
\begin{enumerate}[wide, labelindent=0pt]
    \item[(a)] The \textbf{first upper bounds} on the convergence rate of the gradient method applied to smooth hypoconvex functions for varying step sizes, valid as long as all step sizes stay below some threshold $\tfrac{\bar{h}}{L} \in \big[\tfrac{3}{2L}, \tfrac{2}{L}\big)$ (the expression for $\bar{h} $ depends on the ratio between $\mu$ and $L$). 
    This rate depends on both upper and lower curvatures and continuously interpolates between the results for convex ($\mu=0$) and smooth nonconvex ($\mu=-L$) functions. (Theorem \ref{thm:wc_GM_hypo})
    \item[(b)] The \textbf{tightness} of the above upper bound when all step sizes are below $\tfrac{1}{L}$.  
    (Section \ref{subsec:tightness_gamma_lg_1})
    \item[(c)] A \textbf{tight bound} on the convergence rate of the  gradient method when applied to smooth convex functions $(\mu=0)$ for all step sizes below $\tfrac{3}{2L}$. (Proposition \ref{prop:cvx_rate})
\end{enumerate}
We also explain why all the above rates, which are sublinear, share a common structure, namely they are of the form $\tfrac{C}{q + p N}$ (Theorem \ref{thm:meta_thm_one_step_p_zero_step_q}). 
From these findings, we also deduce:
\begin{enumerate}[wide, labelindent=0pt]
    \item[(d)] The  \textbf{optimal constant step size} recommendation with respect to both $\mu$ and $L$ such that the worst-case upper bound is minimized. Such a recommendation based on both curvatures for hypoconvex functions was not known before. (Proposition \ref{prop:gamma_star})
    \item[(e)] In the case of constant step size $\tfrac{h}{L}$, the \textbf{existence of three worst-case regimes}, namely corresponding to a constant step size belonging to $(0,\tfrac{1}{L}]$, $[\tfrac{1}{L},\tfrac{\bar{h}}{L}]$ and $[\tfrac{\bar{h}}{L},\tfrac{2}{L})$. 
\end{enumerate}
Moreover, following extensive numerical simulations, we conjecture:
\begin{enumerate}[wide, labelindent=0pt]
    \item[(f)] An explicit analytical expression for the \textbf{tight upper bound} on the convergence rate of the gradient method on convex functions with constant step size larger then $\tfrac{3}{2L}$. (Conjecture \ref{conjecture:cvx_large_steps})
    \item[(g)] A partially explicit analytical expression for the \textbf{tight upper bound} corresponding to the above third regime ($\tfrac{h}{L}\in [\tfrac{\bar{h}}{L},\tfrac{2}{L})$)
    (Conjecture \ref{conjecture:3rd_regime})
\end{enumerate}

\paragraph{Limitations.}\label{par:Limitations} Our work focuses on worst-case convergence results, which may in some cases be pessimistic and underestimate actual performance of optimization methods on a given problem. Moreover our rates depend on the knowledge of the curvature constants $L$ and $\mu$, which can sometimes be difficult to estimate, especially the latter. 

\subsection*{Structure of the paper}
Section \ref{sec:Defs_Setup} introduces notations and definitions. Section \ref{sec:PEP_formulation} is dedicated to the formulation of the performance estimation problem for the gradient method applied to hypoconvex functions.  Section \ref{sec:Results_wc_analysis} presents results about the worst-case convergence rates and, as a direct application, the optimal constant step size minimizing the rate is recommended.

	\section{Definitions and setup}\label{sec:Defs_Setup}

Consider the unconstrained optimization problem $\minimize_{x \in \mathbb{R}^d} f(x)$, where $f:\mathbb{R}^d\rightarrow \mathbb{R}$.
\paragraph{Assumption 1.} The function $f$ is smooth and hypoconvex, i.e., its curvature belongs to the interval $\big[\mu, L\big]$, with $L>0$ and $\mu \leq 0$ (see below for precise definitions). 
The step sizes for the iterations are fixed, i.e., decided before the algorithm is run, but not necessarily constant. We denote the step at iteration $i$ by $\tfrac{h_i}{L}$ (i.e., we normalize using the upper curvature $L$).


\subsection{Gradient method and smooth nonconvex functions} \label{subsec:Gradient_Method}
Assume one computes $N$ steps of the gradient method generating the iterations $x_i$ starting from the initial point $x_0$, with $g_i:=\nabla f(x_i)$ the gradient of $f$ with respect to $x_i$,
\begin{align}\label{eq:GM_it_fixed_steps}
    x_{i+1} = x_i - \tfrac{h_i}{L} g_i, \quad \forall i = 0,\dots,N-1.
\end{align}
We denote by $x_*$ and $f_*:=f(x_*)$ any (global) optimal solution and its value (we also have $g_*:=\nabla f(x_*)=0$). 
For all $i\in\mathcal{I}$ (where $\mathcal{I}$ is some index set) we assume we have access to a first-order oracle that provides the triplets $\mathcal{T}:=\big\{(x_i,g_i,f_i)\big\}_{i \in \mathcal{I}} \subseteq \mathbb{R}^d\times\mathbb{R}^d\times\mathbb{R}$, i.e., the iterates, their  gradients and the corresponding function values.
%
%
\paragraph{Convergence} In this work we use the minimum gradient norm of the iterates as our target performance measure, indicating how close we are to a stationary point. It is known that for convex functions the gradient norm decreases after each iteration of the gradient method (provided the step size $h_i$ belongs to $(0,2)$); the proof is given in Appendix \ref{proof:decreasing_gradient}. However this no longer holds for nonconvex or hypoconvex functions, hence we need to adapt our performance measure and use the \emph{minimum gradient norm} among all performed iterations, namely $ \min_{0 {} \leq {} i {} \leq {} N} \big\{\|\nabla f(x_i)\|^2\big\} $.

The state-of-the art rates for $\mu = -L$ 
(i.e., the smooth nonconvex case with $L$-Lipschitz gradient, \cite[Theorem 2]{abbaszadehpeivasti2021GM_smooth} rewritten in an equivalent form) is
{\normalsize
\begin{align} \label{eq:GM_nonconvex_smooth}
    \min_{0 {} \leq {} i {} \leq {} N} \big\{\|\nabla f(x_i)\|^2\big\} {} \leq {} 
        \frac{2L\,\big[f(x_0)-f_*\big]}{1 + 
        \sum\limits_{i=0}^{N-1} \left[2 h_i - \frac{h_i^2}{2} 
        \max(1, h_i)
        \right]}
\end{align}
}%
while for $\mu \rightarrow -\infty$ 
(i.e., when we have no bound on the lower curvature, \cite[Section 1.2.3]{Nesterov_cvx_lectures}) is
{\normalsize
\begin{align} \label{eq:Nesterov_rate}
    \min_{0 {} \leq {} i {} \leq {} N} \big\{\|\nabla f(x_i)\|^2\big\} \leq {} 
        \frac{2L\,\big[f(x_0)-f_*\big]}{1 + 
        \sum\limits_{i=0}^{N-1} \left(2 h_i - h_i^2 
        \right)} \, , \, \forall h_i \in (0,2).
\end{align}
}%
%
%
%
%
\subsection{Smooth hypoconvex functions}\label{subsec:Hypoconvex_func}
The lower and upper curvatures of a function are defined as following.
\begin{definition}\label{def:upper_lower_curvature} 
Let $L>0$ and $\mu > -\infty$. A function $f:\mathbb{R}^d \rightarrow \mathbb{R}$ has an upper curvature $L$ if and only if the function $g:\mathbb{R}^d \rightarrow \mathbb{R}$, $g := \tfrac{L}{2} \|\cdot\|^2 - f$, is convex. Similarly, $f$ has a lower curvature $\mu$ if and only if the function $g:\mathbb{R}^d \rightarrow \mathbb{R}$, $g := f - \tfrac{\mu}{2} \|\cdot\|^2$, is convex.
\end{definition}
%
\begin{definition}\label{eq:Def_class_F} Let $L>0$ and $\mu > -\infty$ such that $\mu \le L$. We denote by $\mathcal{F}_{\mu, L}(\mathbb{R}^d)$ the class of $d$-dimensional smooth functions whose curvature belongs to the interval $\big[\mu, L\big]$.
\end{definition}
\begin{definition}[\cite{taylor_smooth_2017}, Definition 2] A set of triplets $\mathcal{T}:=\big\{\left(x_i, g_i, f_i\right)\big\}_{i \in \mathcal{I}}$, with $x_i, g_i\in \mathbb{R}^d$, $f_i \in \mathbb{R}$, is called $\mathcal{F}_{\mu, L}$-interpolable if and only if there exists a function $f \in \mathcal{F}_{\mu, L}(\mathbb{R}^d)$ such that $\nabla f(x_i)=g_i$ and $f(x_i)=f_i$ for all $i\in \mathcal{I}$.
\end{definition}
Depending on the sign of $\mu$, $f$ is (i) hypoconvex for $\mu<0$, (ii) convex for $\mu=0$ or (iii) strongly-convex for $\mu>0$. The analysis in this paper is restricted to $\mu {} \leq {} 0$. To simplify the notation, we use $\mathcal{F}_{\mu, L}$ and implicitly consider $d$-dimensional real functions.


    \section{Performance estimation of the gradient method on hypoconvex functions}\label{sec:PEP_formulation}
Problem \eqref{eq:PEP_general_inf_dim} below instantiates the general concept of a PEP \eqref{eq:PEP_general_abstract}, using the minimum gradient norm over the iterations as a \textit{performance measure} and the difference between function values as an \textit{initial condition}. Given curvatures $\mu$ and $L$, a bound $\Delta$ characterizing the starting point, and a number of iterations $N$, we solve
\begin{align}\label{eq:PEP_general_inf_dim}
    \begin{aligned}
        \maximize_{f,\,x_0} \quad & \min_{0 \leq i \leq N} \big\{\left\|g_{i}\right\|^2\big\} \\
        \stt \quad & f \in \mathcal{F}_{\mu,L}
        ,\quad f_0-f_* \leq \Delta \qquad\text{ (or, in some cases, $f_{0}-f_{N} \leq \Delta$)}
        \\
        & x_{i+1}=x_{i}-\tfrac{h_{i}}{L} g_{i}, \quad i \in\{0, \ldots, N-1\}
    \end{aligned}
\end{align}
The decision variables of problem \eqref{eq:PEP_general_inf_dim} are function $f$ and starting iterate $x_0$. 
Because the maximum is taken over the entire class of functions $f \in \mathcal{F}_{\mu,L}$, the optimization problem is infinite-dimensional, with an infinite number of constraints (these constraints expressing that function $f$ is hypoconvex). 
To solve it, we relax it and show that the relaxed formulation leads to the same optimal solution. This step is usual in the PEP methodology and is done by restricting the functions to the iterations and use specific interpolation conditions. For smooth hypoconvex functions, these conditions are given in Theorem \ref{thm:interp_hypo_characterization_min}. The PEP solving technique is described in Appendix \ref{appendix:Solve_PEP}. 
%
\begin{theorem} \label{thm:interp_hypo_characterization_min}
Let $\mathcal{T}=\big\{\left(x_i, g_i, f_i\right)\big\}_{i \in \mathcal{I}}$ , $L>0$ and $\mu \in \big(-\infty, L\big]$. Set $\mathcal{T}$ is $\mathcal{F}_{\mu,L}$-interpolable if and only if for every pair of indices $(i,j)$, with $i,j \in \mathcal{I}$ we have
{\normalsize
\begin{align}\label{eq:Interp_hypoconvex}
\hspace{-0.2in}
\begin{aligned}
    f_{i}-f_{j}-\big\langle g_{j}, x_{i}-x_{j}\big \rangle 
    \geq 
    \tfrac{1}{2\big(1-\tfrac{\mu}{L}\big)}\Big(\tfrac{1}{L}\big\|g_{i}-g_{j}\big\|^{2} + 
    \mu\big\|x_{i}-x_{j}\big\|^2-
    2 \tfrac{\mu}{L}\big\langle g_{j}-g_{i}, x_{j}-x_{i}\big\rangle\Big)
\end{aligned}    
\end{align}
}%
Additionally, 
there exists among the set of all interpolating functions one function whose global minimum is finite and is characterized by
\begin{align}\label{eq:Charact_f*}
  f_* = \min_{x\in\mathbb{R}^d} f(x) = 
  \min_{i\in \mathcal{I}} \big\{ f_i - \tfrac{1}{2L} \|g_i\|^2  \big\}
\end{align}
Let 
$
i_* \in \arg \min_{i\in\mathcal{I}} \big\{f_i - \tfrac{1}{2L}\|g_i\|^2\big\};
$
then a global minimizer is given by 
$x_*=x_{i_*} - \tfrac{1}{L}g_{i_*}$.
\end{theorem}
%
\paragraph{\textbf{Interpolation conditions}} The first part of Theorem \ref{thm:interp_hypo_characterization_min} is a direct extension of \cite[Theorem 4]{taylor_smooth_2017} and covers smooth hypoconvex functions besides smooth convex and strongly convex functions. The only difference is allowing negative curvatures due to $\mu<0$. A graphical interpretation of the interpolating function for the particular case $\mu=-L$ is given in \cite[page 71]{PhD_AT_2017}. The interpolation conditions are the main ingredients to prove the convergence rates from Theorem \ref{thm:wc_GM_hypo}.
\paragraph{\textbf{Characterization of the optimal point}} A set of constraints of the type $f_i - f_* \geq 0$ does not guarantee the existence of a function with a global minimum equal to $f_*$. The second part of Theorem \ref{thm:interp_hypo_characterization_min} provides such a guarantee and leads to an exact reformulation after discretization (see Proposition \ref{prop:PEP_equiv} from the Appendix). 
In the particular case $\mu=-L$, the result is given in \cite[Theorem 7]{drori2021complexity} and exploited in \cite{abbaszadehpeivasti2021GM_smooth} to obtain tightness guarantees of the convergence rates for the gradient method. Theorem \ref{thm:interp_hypo_characterization_min} generalizes the previous result to arbitrary lower curvatures $\mu \leq L$, with $L>0$. This extension also appears in \cite[Remark 2.1]{drori2018_primal_interp_func} for strongly convex functions ($\mu>0$).

    \section{Convergence rates}\label{sec:Results_wc_analysis}

One can observe that many of the existing theorems about the convergence of the minimum gradient norm share a similar structure, i.e., that $\min_{0 {} \leq {} i {} \leq {} N} \big\{\|\nabla f(x_i)\|^2\big\} \le \frac{C}{q + \sum_{i=0}^{N-1} p(h_i)}$ holds after 
$N$ iterations, where $C$ and $q$ are constants and $p(h_i)$ is a certain function of the step size at each iteration.

We show that this type of structure can be expected in general for any optimization method, provided some assumption is made on the effect of each individual step, in Theorem \ref{thm:meta_thm_one_step_p_zero_step_q} below, proved in Appendix \ref{proof:thm:meta_thm_one_step_p_zero_step_q}. In addition, when combining with a descent lemma type result, this implies a tighter bound involving the global minimum $f_*$.
\begin{theorem}\label{thm:meta_thm_one_step_p_zero_step_q}
        Assume the effect of each step of some \textit{arbitrary} optimization method on some \textit{arbitrary} class of smooth functions satisfies the following bound 
    \begin{align}\label{eq:one_step_upper_bound}
        \min\{\|\nabla f(x_i)\|^2 \,,\, \|\nabla f(x_{i+1})\|^2\} \leq 
        \frac{f(x_i)-f(x_{i+1})}{p_i}
    \end{align}
    for some constants $p_i$ ($i=0,\ldots,N-1)$, all having the same sign. Then, after $N$ steps the following upper bound holds
    \begin{align}
    \label{eq:meta_thm_only_p_big_rate}
        \min\limits _{0 \leq i \leq N}\{\|\nabla f(x_i)\|^2\} &\leq 
        \frac{f(x_0)-f(x_{N})}{\sum \limits _{i=0}^{N-1} p_i}
    \end{align}
    If in addition we assume that functions in that class satisfy 
    \begin{align}\label{eq:zero_step_upper_bound}
        f(x) - f_* \geq {q} \|\nabla f(x)\|^2 , \forall x \in \mathbb{R}^d
    \end{align}
    for some $q$ with the same sign as $p_i$, we then also have after $N$ steps that
    \begin{align}
    \label{eq:meta_thm_p_and_q_big_rate}
        \min\limits _{0 \leq i \leq N}\{\|\nabla f(x_i)\|^2\} &\leq 
        \frac{f(x_0)-f_*}{q + \sum \limits _{i=0}^{N-1} p_i}
    \end{align}
\end{theorem}
Theorem \ref{thm:meta_thm_one_step_p_zero_step_q} helps to simplify the derivation of upper bounds, since it only requires to perform the analysis of each step separately. Note that the constant in the numerator is the relatively unusual expression $f(x_0)-f(x_N)$, which is always bounded from above by the initial iterate optimality gap $f(x_0)-f_*$. Furthermore, the second part of the Theorem  improves the bound with such a numerator, provided the functions satisfy condition \eqref{eq:zero_step_upper_bound}. 

In the case of hypoconvex functions, Theorem \ref{thm:interp_hypo_characterization_min} shows that it holds with $q=\tfrac{1}{2L}$. Moreover, since all steps of the gradient method are identical (apart from the varying step size $h_i$), it suffices to make the PEP analysis for one step and identify the constant $p_i=p(h_i,\mu,L)$ to obtain a general convergence rate. 

The central result in this section is Theorem \ref{thm:wc_GM_hypo} on the worst-case convergence rates of the gradient method when applied to smooth hypoconvex functions. The result was inferred from the numerical results obtained after solving a large number of PEPs \eqref{eq:PEP_finite_dimensional} for multiple setups of the parameters. We exploited the homogeneity conditions with respect to $L$ and $\Delta$ from \cite[Section 3.5]{taylor_smooth_2017} and fixed $L=\Delta=1$. In this way, it was enough to consider only two unknown parameters: the step sizes $h_i$ and lower curvatures $\mu$ (we use the normalized $\kappa := \tfrac{\mu}{L}$ constant below) and extend the obtained analytical expressions to arbitrary positive $L$ and $\Delta$. Note however that although the numerical solution of PEPs helped us derive this analytical expression, its proof from Appendix \ref{appendix:proof:wc_rate_hypo} is independent. 

\begin{theorem} \label{thm:wc_GM_hypo}
Let $f\in \mathcal{F}_{\mu, L}(\mathbb{R}^d)$ be a smooth hypoconvex function, with $L>0$ and $\mu\in(-\infty,0]$, and let $\kappa:=\tfrac{\mu}{L}$. 
Define step size threshold $\bar{h}(\kappa):=\tfrac{3}{1+\kappa+\sqrt{1-\kappa+\kappa^2}} \in [\tfrac{3}{2}, 2)$ and consider $N$ iterations of the gradient method \eqref{eq:GM_it_fixed_steps} with $h_i \in \big(0,\bar{h}(\kappa)\big]$, \, $i\in\{0,\dots,N-1\}$, generating the sequence $x_1,\dots,x_N$ starting from $x_0$. Then
{\normalsize
\begin{align}\label{eq:GM_hypo_rate_just_p}
    \min_{0 {} \leq {} i {} \leq {} N} \big\{ \|\nabla f(x_i)\|^2\big\} {} \leq {} 
        \frac{2L\big[f(x_0)-f(x_N)\big]}{ 
        \sum\limits_{i=0}^{N-1} p(h_i, \kappa)}
\end{align}
}
where
\vspace{-0.4cm}
{\normalsize
\begin{align}\label{eq:pi_hi_rate_GM}
    p (h_i, \kappa) = 
    \left\{
    \def\arraystretch{2}
        \begin{array}{ll}
            2h_i - h_i^2 \frac{-\kappa}{1-\kappa} &  \text{ if } h_i \in \big(0,1\big] \\
            \frac{h_i(2-h_i)(2-\kappa h_i)}{2-(1+\kappa) h_i} & \text{ if } h_i \in \big[1,\bar{h}(\kappa) \big]
        \end{array}
    \right.
\end{align}
}%
Additionally, if $f$ is bounded from below, then
{\normalsize
\begin{align}\label{eq:GM_hypo_rate_both_p_and_q}
    \min_{0 {} \leq {} i {} \leq {} N} \big\{ \|\nabla f(x_i)\|^2\big\} {} \leq {} 
        \frac{2L\big[f(x_0)-f_*\big]}
        {1+\sum\limits_{i=0}^{N-1} p(h_i, \kappa)}
\end{align}
}%
\end{theorem}

\begin{figure}
    \centering
    \begin{subfigure}[t]{0.49\textwidth}
        \centering
        \includegraphics[width=\textwidth]{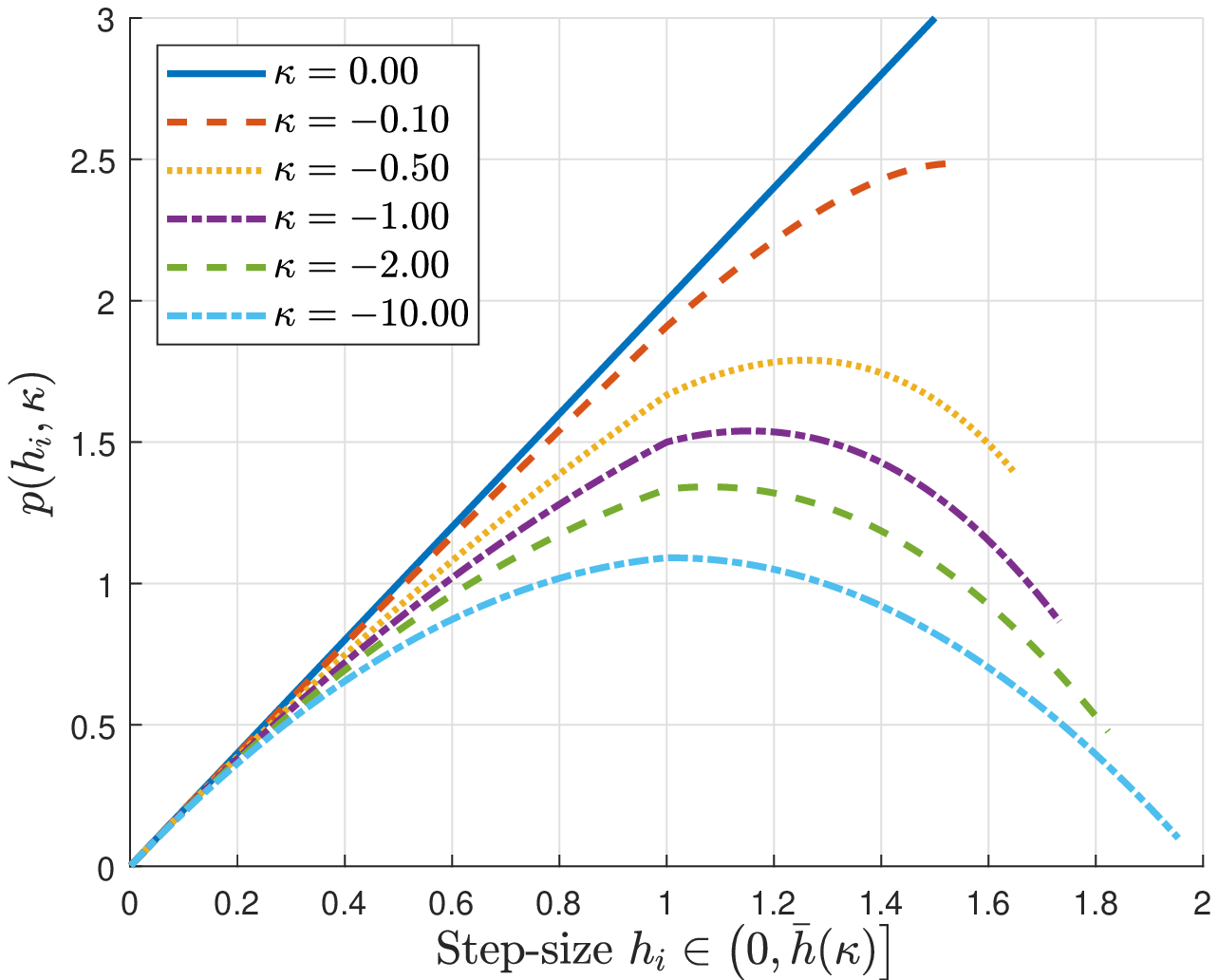}
        \caption{Dependence on the step sizes $h_i$ and the ratio $\kappa$ of the general term \eqref{eq:pi_hi_rate_GM}. For every $\kappa$, the $h_i$ belong to a sub-interval of $(0,2)$; in particular, $h=1$ marks the transition between the two regimes.}
        \label{fig:constant_func_denominator}
    \end{subfigure}
    \hfill
    \begin{subfigure}[t]{0.49\textwidth}
        \centering
        \includegraphics[width=\textwidth]{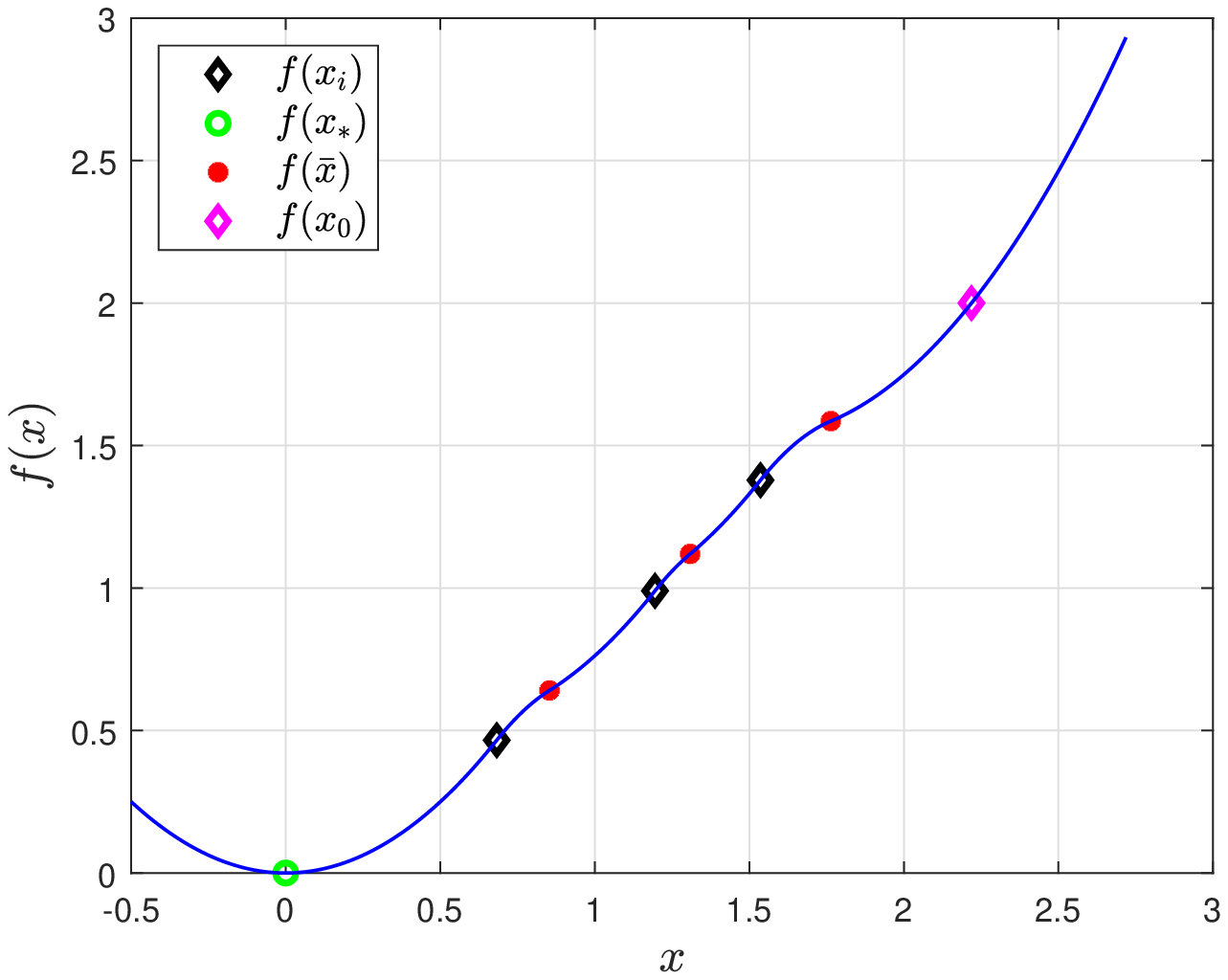}
        \caption{Example of a worst-case function (piecewise quadratic; see Appendix \ref{proof:prop:tightness_low_gamma}), corresponding to step sizes $h_i\leq 1$. Setup: $N=3$, $f_0-f_*=2$, $L=2$, $\kappa=-2$, $h_0=1$, $h_1=0.5$, $h_2=0.75$.}
        \label{fig:func_example}
    \end{subfigure}
    \caption{
    Intuitions about the worst-case regimes valid for varying step sizes.
    }
\end{figure}
To give an intuition about the bounds, the general term $p(h_i,\kappa)$ from \eqref{eq:pi_hi_rate_GM} is depicted in Figure \ref{fig:constant_func_denominator}. A larger value of $p$ translates to faster convergence. Hence one can observe that better rates 
correspond to the regime of step sizes larger than $1$. 

\paragraph{Particular cases} For $\mu=-L$, the rate \eqref{eq:GM_hypo_rate_both_p_and_q} exactly recovers the result from \cite[Theorem 2]{abbaszadehpeivasti2021GM_smooth}, i.e., the \textit{smooth nonconvex} case. Similarly, when $\mu \rightarrow -\infty$, the upper bound becomes the one derived by Nesterov in \cite[Section 1.2.3]{Nesterov_cvx_lectures} based on the descent lemma.

\subsection{Tightness of the results}\label{subsec:tightness_gamma_lg_1}
We now comment on the natural question of the tightness of the above bounds.

\paragraph{Shorter step sizes.} For $h_i \in (0,1]$, the bound is exact. This is proved by constructing a one-dimensional worst-case function example inspired from \cite[Proposition 4]{abbaszadehpeivasti2021GM_smooth}, see Appendix \ref{proof:prop:tightness_low_gamma}. Figure \ref{fig:func_example} shows such a function example, which is piece-wise quadratic.

\begin{proposition}\label{prop:tightness_low_gamma}
    For $h_i \in \big(0,1\big]$, $i\in \big\{0,\dots,N-1\big\}$, the upper bounds from Theorem \ref{thm:wc_GM_hypo} are tight.
\end{proposition}

\paragraph{Larger step sizes} We could not identify a worst-case function valid in the case $h_i \in \big[1, \bar{h}(\kappa)\big]$. However, the primal solution of PEP can be seen as a (numerical) proof of the lower bound. More precisely, for every solution of the optimization problem $\mathcal{T}=\big\{(x_i,g_i,f_i)\big\}_{i\in \mathcal{I}}$, there exists an interpolating function $f$ (see Proposition \ref{prop:PEP_equiv}). From the numerical simulations we observed that when $h > 1$ the worst-case candidate functions for $N$ steps are ($N$+1)-dimensional.

\subsection{The convex case}\label{subsec:rates_cvx}
The class of convex functions can be seen as a particular class of hypoconvex functions with $\mu=0$. The problem of finding exact worst-convergence rates for the gradient norm of the iterates of the gradient method applied to smooth convex functions does not appear to have received a lot of attention. 
The exact rate for the constant step size $h=1$ is determined in \cite[Theorem 5.1]{Kim2021_PEP_Cvx}. With the help of Theorem \ref{thm:wc_GM_hypo}, we extend the upper bounds for step sizes $h_i \in \big(0,\frac{3}{2}\big]$.

The rate from Proposition \ref{prop:cvx_rate} is similar to the one conjectured in \cite[Conjecture 3.1]{drori_performance_2014} for step sizes $h_i{} \leq {} 1$, where the distance to the optimal value $f(x_N)-f_*$ is measured, instead of the gradient norm. The tightness is proved in Appendix \ref{proof:prop:cvx_rate}.
\begin{proposition}[Exact worst-case rate for \textbf{convex} functions] \label{prop:cvx_rate}
Let $f \in \mathcal{F}_{0, L}$ be a smooth convex function and $f_*$ its global minimum. Consider $N$ iterations of the gradient method \eqref{eq:GM_it_fixed_steps} with $h_i \in \big(0, \frac{3}{2}\big]$. Then the upper bounds \eqref{eq:cvx_gL_ll_3/2} are tight.
{\normalsize
\begin{align}\label{eq:cvx_gL_ll_3/2}
        \|\nabla f(x_N)\|^2
        {} \leq {} 
        \frac{L\big[f(x_0)-f(x_N)\big]}{ 
        \sum\limits_{i=0}^{N-1} h_i}
        \quad \text{and} \quad
        \|\nabla f(x_N)\|^2
        {} \leq {} 
        \frac{L\big[f(x_0)-f_*\big]}{\tfrac{1}{2} + 
        \sum\limits_{i=0}^{N-1} h_i}
\end{align}
}%
\end{proposition}
Following extensive numerical simulations, we conjecture the following upper bound for constant step sizes larger than $\tfrac{3}{2}$.

\begin{conjecture}\label{conjecture:cvx_large_steps}
Let $f \in \mathcal{F}_{0, L}$ be a smooth convex function and $f_*$ its global minimum. Consider $N$ iterations of the gradient method with a constant step size $h \in \big(\frac{3}{2},2\big)$. Then
{\normalsize
\begin{align*}
\begin{aligned}
    \|\nabla f(x_N)\|^2
    {} \leq {}  
    \tfrac{2L\big[f(x_0)-f(x_N)\big]}
    {\min \big\{
    -1+(1-h)^{-2N}
    \,,\,
    {
        2 N h} 
         \big\}}
    \quad \text{and} \quad
    \|\nabla f(x_N)\|^2
    {} \leq {}  
    \tfrac{2L\big[f(x_0)-f_*\big]}
    {\min \big\{
    (1-h)^{-2N}
    \,,\,
    {1 + 
        2 N h} 
         \big\}}.
\end{aligned}
\end{align*}
}%
\end{conjecture}

\subsection{Application: optimal constant step size}\label{subsec:opt_gamma}
A direct benefit of the upper bounds from Theorem \ref{thm:wc_GM_hypo} is that they allow to deduce an optimal constant step size that minimizes the worst-case convergence rate of the gradient method applied to smooth hypoconvex functions. Since it exploits the lower curvature information better, this recommendation is superior to that of the smooth nonconvex case.

\begin{proposition}\label{prop:gamma_star}
    Let $f \in \mathcal{F}_{\mu,L}$ be a smooth hypoconvex function with $L>0$ and $\mu \leq 0$, $\kappa=\tfrac{\mu}{L} \leq 0$ and $\bar{\kappa}:= \tfrac{-9-5\sqrt{5} + \sqrt{190+90\sqrt{5}}}{4}
    \approx -0.1001$.
    Then the optimal step size $h_*$ for the gradient method with respect to the worst-case convergence rate from Theorem \ref{thm:wc_GM_hypo} is
    {\normalsize
    \begin{align}\label{eq:opt_stepsize_min_known_bound}
    h_*(\kappa) = \left\{ 
    \arraycolsep=3pt
    \def\arraystretch{1.2}
    \begin{array}{ll}
        h_{\text{opt}} & \kappa {} \leq {} \bar{\kappa} \\
        \bar{h}(\kappa) & \bar{\kappa} < \kappa \leq 0
    \end{array}
    \right.
    \end{align}
    }%
    where $\bar{h}(\kappa)$ is defined in Theorem \ref{thm:wc_GM_hypo} and
    $h_{\text{opt}}$ is the unique solution in $\big[1, \bar{h}(\kappa)\big]$ of 
    {\normalsize
    \begin{align*}
    \begin{aligned}
        -\kappa\left(1+\kappa\right) h^3 +
    \big[3\kappa + \left(1+\kappa\right)^2\big] h^2 -
    4\left(1+\kappa\right) h +
    4 = 0.
    \end{aligned}
    \end{align*}
    }%
\end{proposition}

The proof of Proposition \ref{prop:gamma_star} is given in the Appendix \ref{proof:prop:gamma_star}.
When taking the limit $\kappa \rightarrow -\infty$, from \eqref{eq:opt_stepsize_min_known_bound} we get the ``classical optimal step size'' $h=1$. Likewise, for $\kappa=-1$ we obtain the optimal step size for smooth nonconvex function given in \cite[Theorem 3]{abbaszadehpeivasti2021GM_smooth}, namely $h=\tfrac{2}{\sqrt{3}}$.

One may be interested in the actual benefit of using our recommendation. To answer this question, we focus on constant step sizes $h$, and compare the (inverse of the) constant $p(h,\kappa)$ for different choices of $h$. Figure \ref{fig:Compare_opt_ct} shows $1/p(h,\kappa)$ for: (i) the ``classical optimal step'' $h=1$, (ii) the optimal step size for smooth functions $h_*^s=\tfrac{2}{\sqrt{3}}$, (iii) the optimal step size for smooth convex functions $h=\tfrac{3}{2}$ (Proposition \ref{prop:gamma_star} with $\kappa=0$), and (iv) the optimal step size from Proposition \ref{prop:gamma_star}. When the nonconvexity decreases (i.e., when $\kappa$ becomes closer to zero), the optimal step size from \eqref{eq:opt_stepsize_min_known_bound} provides a larger improvement in comparison to the step $h_*^s$, that does not benefit from the lower curvature information. In the limit $\kappa \rightarrow \infty$, the classical optimal step size $h=1$ is recovered. Nevertheless, the result shows the maximal guaranteed improvement one can get from optimizing the constant, and it provides a continuous interpolation between the optimal step sizes for $\kappa=-1$ and $\kappa=0$.

Figure \ref{fig:Compare_gamma_max_opt_vs_k} presents the optimal constant step size $h_*(\kappa)$ that maximizes the $p$ term from \eqref{eq:pi_hi_rate_GM} and the threshold $\bar{h}(\kappa)$ of the rate from Theorem \ref{thm:wc_GM_hypo}. The non-smooth part which appears for $\kappa > \bar{\kappa}$ is due to the threshold $\bar{h}(\kappa)$. For this range of close-to-convex functions only this suggests that the optimal step size belongs to the range of large steps, $h > \bar{h}(\kappa)$, which is not covered by Theorem \ref{thm:wc_GM_hypo}.
\begin{figure}
    \centering
    \begin{subfigure}[t]{0.49\textwidth}
        \centering
        \includegraphics[width=\textwidth]{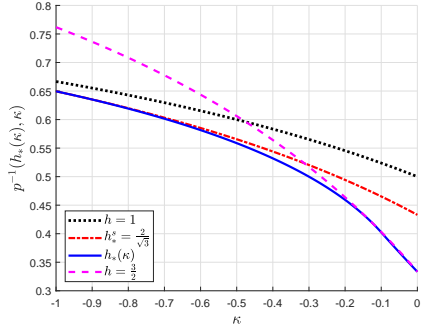}
        \caption{Comparison between the 
    ``classical optimal step'' $h=1$ 
    (in \textit{black})
    , 
    the step $h_*^s=\tfrac{2}{\sqrt{3}}$ recommended in \cite{abbaszadehpeivasti2021GM_smooth} for smooth nonconvex functions,
    the optimal step $h=\tfrac{3}{2}$ for convex functions
    and the optimal step size recommendation $h_*(\kappa)$ from Proposition \ref{prop:gamma_star},
    in terms of the inverse one-step bounds $p(h_i,\kappa)$ from \eqref{eq:pi_hi_rate_GM} for $h_i \geq 1$.}
        \label{fig:Compare_opt_ct}
    \end{subfigure}
    \hfill
    \begin{subfigure}[t]{0.49\textwidth}
        \centering
        \includegraphics[width=\textwidth]{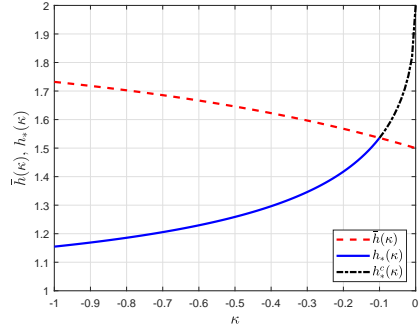}
        \caption{The threshold $\bar{h}(\kappa)$ from Theorem \ref{thm:wc_GM_hypo} 
        and the optimal step size $h_*(\kappa)$ from Proposition \ref{prop:gamma_star}. 
    For $\kappa > \bar{\kappa} \approx -0.1$, the step size maximizing the $p$ term \eqref{eq:pi_hi_rate_GM} is the maximum allowed one, $\bar{h}(\kappa)$. The \textit{black} dashed line marks the conjectured optimal step size $h_*^c(\kappa)$ for $\kappa > \bar{\kappa}$ and a large number of iterations $N$, following Proposition \ref{prop:Conjectured_optimal_step}.}
        \label{fig:Compare_gamma_max_opt_vs_k}
    \end{subfigure} \\
    \begin{subfigure}[t]{0.49\textwidth}
        \centering
        \includegraphics[width=\textwidth]{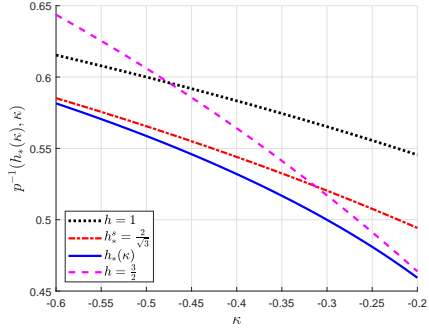}
        \caption{Zoom on a section of (a)}
        \label{fig:Compare_opt_ct_zoom}
    \end{subfigure}
    \hfill
    \begin{subfigure}[t]{0.49\textwidth}
        \centering
        \includegraphics[width=\textwidth]{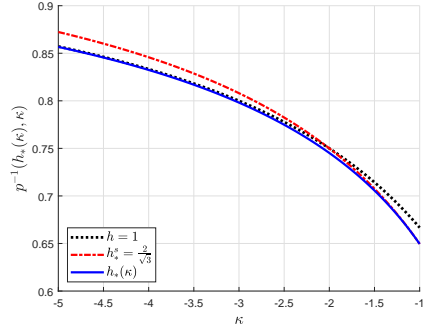}
        \caption{For $\kappa < -1$ the optimal step decreases to $h=1$}
        \label{fig:Compare_opt_ct_k_ll_m1}
    \end{subfigure}
    \caption{Step-sizes recommendations and their benefit}
    \label{fig:Compare_opt}
\end{figure}
\subsection{On the regime with larger step sizes} \label{subsec:3rd_regime}
We derived analytical expressions for the first two step size regimes in  Theorem \ref{thm:wc_GM_hypo}, but a closed form of the third regime, with steps $h \in \big(\bar{h}(\kappa), 2\big)$, seems to be more difficult to obtain.
However, using the numerical results from the solution of PEP \eqref{eq:PEP_finite_dimensional} (see Appendix \ref{appendix:Solve_PEP}) for multiple setups, the following partially explicitly analytical expressions of the worst-case convergence rate for constant step sizes are conjectured.

\begin{conjecture}\label{conjecture:3rd_regime} Let $f \in \mathcal{F}_{\mu,L}$ be a smooth hypoconvex function, $f_*$ its global minimum and $x_0$ the starting point. Consider $N$ iterations of the gradient method \eqref{eq:GM_it_fixed_steps} with a constant step $h \in \big(\bar{h}(\kappa), 2\big)$. Then
{\normalsize
\begin{align*}
\begin{aligned}
     \min_{0 {} \leq {} i {} \leq {} N} & \big\{\|\nabla f(x_i)\|^2 \big\}
    {} \leq {}
     \frac{2 L \big[f(x_0)-f(x_N)\big]}
     {\min \Big\{
     {-1+(1-h)^{-2N}} \ , \
     -1+r(L,\kappa,h) + N \frac{h(2-h)(2-\kappa h)}{2-h \left(1+\kappa \right)}
     \Big\}} \\
     \min_{0 {} \leq {} i {} \leq {} N} & \big\{\|\nabla f(x_i)\|^2 \big\}
    {} \leq {}
     \frac{2 L \big[f(x_0)-f_*\big]}
     {\min \Big\{
     {(1-h)^{-2N}} \ , \
     r(L,\kappa,h) + N \frac{h(2-h)(2-\kappa h)}{2-h \left(1+\kappa \right)}
     \Big\}}
\end{aligned}
\end{align*}
}
where for $r=r(L,\kappa,h)$ is an unknown function that does \emph{not} depend on the number of iterations $N$.
\end{conjecture}
These conjectured expressions involve the maximum of two rates, one of whom is of the usual sublinear type. From those, one can infer that there exists a number of steps $N_0$ such that, when performing $N>N_0$ iterations, the first $N_0$ steps belong to the first rate and the next ones to the second rate. For the latter, the fraction multiplying $N$ is the same as in \eqref{eq:pi_hi_rate_GM} for $h_i \geq 1$. 

\paragraph{\textbf{Optimal step size}} Asymptotically, for a large number of iterations $N$, the constant term $r(L,\kappa,h)$ becomes negligible. Hence minimizing the upper bound leads to an optimization problem over the entire range $h \in \big(0,2\big)$. Therefore, assuming Conjecture \ref{conjecture:3rd_regime} we propose Proposition \ref{prop:Conjectured_optimal_step} where the threshold $\bar{\kappa}$ from Proposition \ref{prop:gamma_star} disappears and the recommendation covers all regimes. The proof (assuming Conjecture \ref{conjecture:3rd_regime}) is given in Appendix \ref{proof:prop:Conjectured_optimal_step}.
\begin{proposition}\label{prop:Conjectured_optimal_step}
Let $f \in \mathcal{F}_{\mu,L}$ be a smooth hypoconvex function with $L>0$ and $\mu < 0$, and $\kappa=\tfrac{\mu}{L} < 0$. Then the optimal step size $h_*$ that minimizes the asymptotic worst-case convergence rate of the gradient method applied to hypoconvex functions is the unique solution in $[1, 2)$ of 
    {\normalsize
    \begin{align*}
    \begin{aligned}
        -\kappa\left(1+\kappa\right) h^3 +
    \left[3\kappa + \left(1+\kappa\right)^2\right] h^2 -
    4\left(1+\kappa\right) h +
    4 = 0
    \end{aligned}
    \end{align*}
    }%
\end{proposition}%
Figure \ref{fig:Compare_gamma_max_opt_vs_k} shows this conjectured extension of the optimal step size $h_*^c(\kappa)$ in \textit{black} dashed line. In the limit $\kappa \rightarrow 0$, the optimal step size for the convex case $\kappa=0$ is obtained as the upper bound of the convergence interval.

        
    \section{Conclusion}\label{sec:conclusion}%

%

This paper performs the first worst-case analysis of the gradient method with varying step sizes applied to  smooth hypoconvex functions,  using the performance estimation technique. Interpolation conditions for functions from this class, and characterization of their minimizer, are introduced in Theorem \ref{thm:interp_hypo_characterization_min}. We identify three regimes with respect to the range of step sizes. For two of them a proof of the upper bound for step sizes $h \leq \bar{h}(\kappa)$ is provided. Its tightness is shown by constructing a worst-case example for steps $h\leq 1$. For larger steps, i.e., $h\in[1,\bar{h}(\kappa)]$, solving the PEP and identifying the analytical expression translates into an implicit numerical proof. For the third regime  $h \in \big(\bar{h}(\kappa),2\big)$, we conjecture a partially explicit analytical rate. We also obtain tight convergence rates for convex functions for the steps range $h \in \big(0,\tfrac{3}{2}\big]$. Finally, as a direct application of our rates, we identify an optimal constant step size recommendation for the gradient method on hypoconvex functions.  


    		
    
	
	\bibliographystyle{plain}
	\bibliography{ms.bib}
    
    
    \clearpage
    
    \begin{appendix}%
        \section{Proof of the gradient norm monotonicity in the convex case}\label{proof:decreasing_gradient}

 
\begin{proof}[Monotonicity of gradient norm for convex functions]\label{proof:decreasing_gradient_norm_cvx_func} 
We show that the gradient norm is not increasing after performing one step of the gradient method with step size $\tfrac{h}{L}$, where $h \in (0,2)$, on a smooth convex function $f \in \mathcal{F}_{0,L}$. In other words, given $x_1=x_0 - \tfrac{h}{L} g_0$ (recall we define $g_i=\nabla f(x_i)$), we have $\|g_0\|^2 \ge \|g_1\|^2  $. 

From cocoercivity of the gradient \cite[Theorem 2.1.5]{Nesterov_cvx_lectures} we have
\begin{align*}
    \langle g_1 - g_0 ,\, x_1 - x_0 \rangle \geq \frac{1}{L} \|g_0-g_1\|^2.
\end{align*}
Using $x_1=x_0 - \tfrac{h}{L} g_0$ in this inequality gives
\begin{align*}
    -\frac{h}{L}\langle g_1 - g_0 ,\, g_0 \rangle \geq \frac{1}{L} \|g_0-g_1\|^2 
    \Leftrightarrow
    \langle g_0 - g_1 ,\, g_0 \rangle \geq \frac{1}{h} \|g_0-g_1\|^2.
\end{align*}
The left-hand side is equal to $\|g_0\|^2 - \langle g_1, g_0\rangle$, and the inner product can be written as a difference of squares \[ \langle g_1, g_0\rangle = \frac{1}{2} \|g_0\|^2 + \frac{1}{2} \|g_1\|^2 - \frac{1}{2}\|g_0-g_1\|^2   \]
allowing us to rewrite the inequality as 
\begin{align*}
    \frac{1}{2} \|g_0\|^2 - \frac{1}{2} \|g_1\|^2 + \frac{1}{2}\|g_0-g_1\|^2  \geq \frac{1}{h} \|g_0-g_1\|^2 
\end{align*}
which then gives
    \[
    \|g_0\|^2 - \|g_1\|^2 \geq \frac{2-h}{h} \|g_0 - g_1\|^2.\]
Therefore we have that the gradient norm is non-increasing when $h\in(0,2)$.

\end{proof}

        \clearpage
        \section{Proof of interpolation conditions for smooth hypoconvex functions} \label{appendix:proof_thm_interp_charact_min}

To build the proof of Theorem \ref{thm:interp_hypo_characterization_min} we introduce some auxiliary results. Upper and lower bounds characterizing smooth hypoconvex functions are given in Lemma \ref{lemma:hypoconvex_smooth_quad_bounds}. They are used to construct Lemma \ref{lemma:hypconvex_mu_L_interpolable} showing the effect on interpolability conditions when shifting the curvature.

\begin{lemma}[Quadratic bounds]\label{lemma:hypoconvex_smooth_quad_bounds} 
If $f\in \mathcal{F}_{\mu, L}$, with $L > 0$ and $\mu \in \big(-\infty, L\big]$, then
{\normalsize
\begin{align}\label{eq:def_hypoconvex_smooth}
    \frac{\mu}{2} \|x-y\|^2 {} \leq {} f(x) - f(y) - \langle \nabla f(y), x-y \rangle {} \leq {} \frac{L}{2} \|x-y\|^2 \quad \forall x, y \in \mathbb{R}^d.
\end{align}
}%
\end{lemma}
\begin{proof}
The bounds result by inserting the following known inequality of smooth convex functions,
\begin{align*}
    f(x) {} \geq {} f(y) + \left \langle \nabla f(y), x-y \right \rangle, \, \forall x,y \in \mathbb{R}^d,
\end{align*}
in Definition \ref{def:upper_lower_curvature} of curvatures.
\end{proof}
\begin{lemma}\label{lemma:hypconvex_mu_L_interpolable} Consider a set $\big\{\big(x_i, g_i, f_i\big)\big\}_{i \in \mathcal{I}}$. For any constants $\mu$, $L$ with $\mu \in \big(-\infty\,,\, L\big]$, $0 < L {}\leq{} \infty$, the following propositions are equivalent:
\begin{enumerate}
    \item $\big\{\big(x_i, g_i, f_i\big)\big\}_{i \in \mathcal{I}}$ if $\mathcal{F}_{\mu, L}$-interpolable,
    \item $\Big\{\big(x_i, g_i - \mu x_i, f_i - \frac{\mu}{2} \|x_i\|^2 \big)\Big\}_{i \in \mathcal{I}}$ if $\mathcal{F}_{0, L-\mu}$-interpolable.
\end{enumerate}
\end{lemma}
\begin{proof}
This lemma is a direct extension of \cite[Lemma1]{taylor_smooth_2017} in the sense of including the hypoconvex functions, i.e., $\mu<0$. The idea is to use minimal curvature subtraction and write Lemma \ref{lemma:hypoconvex_smooth_quad_bounds} for  $h(x):=f(x)-\frac{\mu}{2} \|x\|^2$. Then $\nabla h(x)=\nabla f(x)-\mu x$ and from \eqref{eq:def_hypoconvex_smooth} we have the following equivalent inequalities:
{\normalsize
\begin{align*}
    \frac{\mu}{2} \|x-y\|^2 {}\leq{}& h(x) - h(y) - \langle \nabla h(y), x-y \rangle {}\leq{} \frac{L}{2} \|x-y\|^2 \nonumber \\
    \frac{\mu}{2} \|x-y\|^2 {}\leq{}& f(x) - f(y)- \langle \nabla f(y), x-y \rangle + \frac{\mu}{2} \|x-y\|^2 {}\leq{} \frac{L}{2} \|x-y\|^2 \nonumber \\
    0 {}\leq{}& f(x) - f(y)- \langle \nabla f(y), x-y \rangle {}\leq{} \frac{L-\mu}{2} \|x-y\|^2
\end{align*}
}%
Hence, $\big\{\big(x_i, g_i, f_i\big)\big\}_{i \in \mathcal{I}}$ is $\mathcal{F}_{\mu,L}$-interpolable if and only if $\Big\{\big(x_i, g_i-\mu x, f_i-\tfrac{\mu}{2}\|x_i\|^2\big)\Big\}_{i \in \mathcal{I}}$ is $\mathcal{F}_{0, L-\mu}$-interpolable. 
\end{proof}
From \cite{taylor_smooth_2017} we recall the following results.
\begin{theorem}[\cite{taylor_smooth_2017}, Theorem 1] \label{thm:taylor_cvx_interp} 
\textbf{(Convex interpolation)} 
The set $\big\{\big(x_i, g_i, f_i\big)\big\}_{i \in \mathcal{I}}$ is $\mathcal{F}_{0,\infty}$-interpolable if and only if 
{\normalsize
\begin{align*}
    f_i - f_j - \langle g_j, x_i-x_j \rangle {}\geq{} 0 \quad i,j\in \mathcal{I}.
\end{align*}
}%
\end{theorem}
\begin{lemma}[\cite{taylor_smooth_2017}, Lemma 2]\label{lemma:convex_L_conjugate}
Consider a set $\left\{\big(x_{i}, g_{i}, f_{i}\big)\right\}_{i \in I}$.
The following propositions are equivalent $\forall L: 0<L {}\leq{}+\infty$:
\begin{enumerate}
    \item $\big\{\big(x_{i}, g_{i}, f_{i}\big)\big\}_{i \in I}$ is $\mathcal{F}_{0, L}$-interpolable,
    \item $\Big\{\big(g_{i}, x_{i},\left\langle g_{i}, x_{i}\right\rangle-f_{i}\big)\Big\}_{i \in I}$ is $\mathcal{F}_{1 / L, \infty}$-interpolable.
\end{enumerate}
\end{lemma}
We prove now \textbf{Theorem \ref{thm:interp_hypo_characterization_min}}, based on the same steps as the one of \cite[Theorem 4]{taylor_smooth_2017}, which is more detailed.
\begin{proof}[Proof of Theorem \ref{thm:interp_hypo_characterization_min}]\label{thm_proof:interp_hypo_characterization_min}
We divide it in two parts: (i) proving the interpolation conditions and (ii) proving the characterization of the global minimum $f_*$ \eqref{eq:Charact_f*}.

\paragraph{(i) Proof of the \textbf{interpolation conditions} \eqref{eq:Interp_hypoconvex}.} 

Taylor et. al. state in \cite[Theorem 4]{taylor_smooth_2017} the interpolation conditions of smooth strongly-convex functions $f \in \mathcal{F}_{\mu,L}$, with $\mu>0$. The inequalities from Lemma \ref{lemma:hypoconvex_smooth_quad_bounds} are valid for any finite $\mu {}\leq{} L$, therefore we have the same series of equivalences from \cite[Theorem 4]{taylor_smooth_2017}:

\begin{enumerate}[label=(\alph*)]
    \item $\Big\{\big(x_{i}\, , \, g_{i}\, , \, f_{i}\big)\Big\}_{i \in I}$ is $\mathcal{F}_{\mu, L}$-interpolable,
    
    \item $\Big\{\big(x_{i}\, , \, g_{i}-\mu  x_{i}\, , \, f_{i}-\frac{\mu}{2}\|x_{i}\|^{2}\big)\Big\}_{i \in I}$ is $\mathcal{F}_{0, L-\mu}$-interpolable,
    
    \item $\Big\{\big(g_{i}-\mu  x_{i}\, , \, x_{i}\, , \,\langle g_{i}, x_{i}\rangle-f_{i}-\frac{\mu}{2}\|x_{i}\|^{2}\big)\Big\}_{i \in I}$ is $\mathcal{F}_{1 /(L-\mu), \infty}$-interpolable,

    \item $\Big\{\Big(g_{i}-\mu  x_{i}\, , \, \frac{L x_{i}}{L-\mu}-\frac{g_{i}}{L-\mu}\, , \, \frac{L\langle g_{i}, x_{i}\rangle}{L-\mu}-f_{i}-\frac{\mu L\|x_{i}\|^{2}}{2(L-\mu)}-\frac{\|g_{i}\|^{2}}{2(L-\mu)}\Big)\Big\}_{i \in I}$
is $\mathcal{F}_{0, \infty}$-interpolable,
    
    \item $\Big\{\Big(\frac{L x_{i}}{L-\mu}-\frac{g_{i}}{L-\mu}\, , \, g_{i}-\mu x_{i}\, , \, \frac{\mu \langle g_{i}, x_{i}\rangle}{L-\mu}+f_{i}-\frac{\mu L\|x_{i}\|^{2}}{2(L-\mu)}-\frac{\|g_{i}\|^{2}}{2(L-\mu)}\Big)\Big\}_{i \in I}$
is $\mathcal{F}_{0, \infty}$-interpolable.
\end{enumerate}
The equivalences: (a) $\Leftrightarrow$ (b) and (c) $\Leftrightarrow$ (d) come from direct applications of Lemma \ref{lemma:hypconvex_mu_L_interpolable}. The equivalences: (b) $\Leftrightarrow$ (c) and (d) $\Leftrightarrow$ (e) come from direct applications of Lemma \ref{lemma:convex_L_conjugate}. Therefore, (a) and (e) are equivalent. Consequently, the interpolation conditions \eqref{eq:Interp_hypoconvex} follow directly by applying Theorem \ref{thm:taylor_cvx_interp} on the $\mathcal{F}_{0,\infty}$-interpolable set (e).

\paragraph{(ii) Proof of \textbf{optimal point characterization} \eqref{eq:Charact_f*}.} 
We adapt the steps from the proof of \cite[Theorem 7]{drori2021complexity} on smooth nonconvex functions to arbitrary lower curvature $\mu$. Consider the function
{\normalsize
\begin{align*}
    \begin{aligned}
        Z(y) {}:={} \min_{\alpha \in \Delta_{\mathcal{I}}} 
        &\Bigg\{
            \tfrac{L-\mu}{2} 
            \big\|
            y - \sum\limits_{i \in \mathcal{I}} \alpha_i
             \big[x_i - \tfrac{1}{L-\mu} \big(g_i - \mu x_i \big) \big] 
            \big\|^2 + 
            \\ & \qquad
            \sum\limits_{i \in \mathcal{I}} \alpha_i
             \Big( f_i - \tfrac{\mu}{2} \|x_i\|^2 - \tfrac{1}{2(L-\mu)} \|g_i-\mu x_i\|^2 \Big)
        \Bigg\}
    \end{aligned}
\end{align*}
}%
where $\Delta_{\mathcal{I}}$ is the $n$ dimensional unit simplex (with $n:=|\mathcal{I}|$):
\begin{align*}
    \Delta_{{\mathcal{I}}} := \Big\{ {\alpha} \in \mathbb{R}^n: \sum\limits_{i\in\mathcal{I}} \alpha_i = 1, \alpha_i {}\geq{} 0, \forall i \in 0,\dots,n-1 \Big\}.
\end{align*}
The function $Z$ is the primal interpolation function $W_{\mathcal{T}}^{C}$ as defined in \cite[Definition 2.1]{drori2018_primal_interp_func}. One can check this by replacing in the definition $C \leftarrow \{0\}$, $L \leftarrow L-\mu$ and 
\begin{align*}
    \mathcal{T} \leftarrow \Big\{\big(x_i, g_i-\mu x_i, f_i - \tfrac{\mu}{2} \|x_i\|^2\big) \Big\}_{i \in \mathcal{I}}.
\end{align*}
The set $\mathcal{T}$ is $\mathcal{F}_{0, L-\mu}$-interpolable. Hence, from \cite[Theorem 1]{drori2018_primal_interp_func} it follows that $Z$ is convex, with upper curvature $L-\mu$, and satisfies
{\normalsize
\begin{align*}
    \begin{aligned}
        Z(x_i) = f_i - \frac{\mu}{2} \|x_i\|^2 \quad \text { and } \quad
        \nabla Z(x_i) = g_i - \mu x_i.
    \end{aligned}
\end{align*}
}%
We proceed to define the function
{\normalsize
\begin{align*}
    \hat{W}(y) := Z(y) + \frac{\mu}{2} \|y\|^2,
\end{align*}
}%
which belongs to the function class $\mathcal{F}_{\mu,L}$ (see Lemma \ref{lemma:hypconvex_mu_L_interpolable}) and satisfies 
{\normalsize
\begin{align*}
    \begin{aligned}
        \hat{W}(x_i) = f_i \quad \text { and } \quad
        \nabla \hat{W}(x_i) = g_i.
    \end{aligned}
\end{align*}
}%
Using algebraic manipulations, we can express $\hat{W}$ as
{\normalsize
\begin{align}\label{eq:hat_W_equiv}
\begin{aligned}
    \hat{W}(y) {}={} \min_{\alpha \in \Delta_{\mathcal{I}}} 
            &\Bigg\{
                \tfrac{L}{2}
                \big\|
                y - \sum\limits_{i \in \mathcal{I}} \alpha_i
                 \big(x_i - \tfrac{1}{L}g_i \big) 
                \big\|^2 + 
                \tfrac{L}{2}\tfrac{\kappa}{1-\kappa} \big\|\sum\limits_{i \in \mathcal{I}} \alpha_i
                 \big(x_i - \tfrac{1}{L}g_i \big) \big\|^2 +
                \\ & \qquad 
                \sum\limits_{i \in \mathcal{I}} \alpha_i
                 \Big( f_i - \tfrac{1}{2L} \|g_i\|^2 - \tfrac{L}{2} \tfrac{\kappa}{1-\kappa} \big\|x_i - \tfrac{1}{L} g_i \big\|^2 \Big)
            \Bigg\}.
\end{aligned}            
\end{align}
}%
Further, we \textit{lower bound} the first squared norm by $0$ and use the convexity of the squared norm for the second inequality:
{\normalsize
\begin{align}\label{eq:W_hat_upper_bound}
    \begin{aligned}
        \hat{W}(y) {}{}\geq{}{} & 
            \min_{\alpha \in \Delta_{\mathcal{I}}} 
            \Bigg\{
                \tfrac{L}{2}\tfrac{\kappa}{1-\kappa} \big\|\sum\limits_{i \in \mathcal{I}} \alpha_i
                 \big(x_i - \tfrac{1}{L}g_i \big) \big\|^2 +
                \sum\limits_{i \in \mathcal{I}} \alpha_i
                 \Big( f_i - \tfrac{1}{2L} \|g_i\|^2 - \tfrac{L}{2} \tfrac{\kappa}{1-\kappa} \big\|x_i - \tfrac{1}{L} g_i \big\|^2 \Big)
            \Bigg\} \\
            {}\geq{} & 
            \min_{\alpha \in \Delta_{\mathcal{I}}} 
            \Bigg\{
                \tfrac{L}{2}\tfrac{\kappa}{1-\kappa} \sum\limits_{i \in \mathcal{I}} \alpha_i\big\|
                 x_i - \tfrac{1}{L}g_i \big\|^2 +
                \sum\limits_{i \in \mathcal{I}} \alpha_i
                 \Big( f_i - \tfrac{1}{2L} \|g_i\|^2 - \tfrac{L}{2} \tfrac{\kappa}{1-\kappa} \big\|x_i - \tfrac{1}{L} g_i \big\|^2 \Big)
            \Bigg\} \\
            {}={}&
            \min_{\alpha \in \Delta_{\mathcal{I}}} 
            \Bigg\{
                \sum\limits_{i \in \mathcal{I}} \alpha_i
                 \Big( f_i - \tfrac{1}{2L} \|g_i\|^2 \Big)
            \Bigg\} \\
            {}={}&
            \min_{i \in \mathcal{I}}\,\,\,\,
                 \Big\{ f_i - \tfrac{1}{2L} \|g_i\|^2 \Big\} \\
            {}={}&
            f_{i_*} - \tfrac{1}{2L} \|g_{i_*}\|^2.
    \end{aligned}
\end{align}
}%
For the \textit{upper bound}, in \eqref{eq:hat_W_equiv} we take $y:=x_{i_*}-\tfrac{1}{L}g_{i_*}$  and ${\alpha}={e}_{i_*}$ (the ${i_*}$-th unit vector):
{\normalsize
\begin{align}\label{eq:hat_W_lower_bound}
\begin{aligned}
    \hat{W}\big(x_{i_*}-\tfrac{1}{L}g_{i_*}\big) 
                {}&\leq{}
                \tfrac{L}{2}\tfrac{\kappa}{1-\kappa} \big\|x_{i_*} - \tfrac{1}{L}g_{i_*} \big\|^2 +
                 f_{i_*} - \tfrac{1}{2L} \|g_{i_*}\|^2 - \tfrac{L}{2} \tfrac{\kappa}{1-\kappa} \big\|x_{i_*} - \tfrac{1}{L} g_{i_*} \big\|^2 \\
                 {}&={}
                 f_{i_*} - \tfrac{1}{2L} \|g_{i_*}\|^2.
\end{aligned}            
\end{align}
}%
Therefore, from \eqref{eq:W_hat_upper_bound} and \eqref{eq:hat_W_lower_bound} it follows
$\hat{W}\big(x_{i_*}-\tfrac{1}{L}g_{i_*}\big) = f_{i_*} - \tfrac{1}{2L} \|g_{i_*}\|^2$.
\end{proof}
    	
    	\clearpage
        \section{Solving the PEP}\label{appendix:Solve_PEP}
Based on Theorem \ref{thm:interp_hypo_characterization_min}, the infinite-dimensional PEP \eqref{eq:PEP_general_inf_dim} is transformed into one of the finite-dimensional PEPs \eqref{eq:PEP_finite_dimensional_no_opt} or \eqref{eq:PEP_finite_dimensional}, where $\mathcal{I}=\{0,\dots,N\}$ and $\mathcal{I}^*=\{0,\dots,N,*\}$, respectively. The two finite-dimensional PEPs correspond to the different initial conditions: $f_0 - f_N \leq \Delta$ and $f_0 - f_* \leq \Delta$, respectively. The procedure of transforming the original PEP into a tractable version is similar for both formulations.

\paragraph{Initial condition $f_0 - f_N \leq \Delta$.} In this case, the discretized PEP is
\begin{align}\label{eq:PEP_finite_dimensional_no_opt}
    \begin{aligned}
        \maximize_{\{(x_i,g_i,f_i)\}_{i \in \mathcal{I}}} \quad & \min _{0 \leq i \leq N} \big\{\left\|g_{i}\right\|^2\big\} \\
        \stt \quad & \big\{(x_i,g_i,f_i)\big\}_{i \in \mathcal{I}} \text{  satisfy \eqref{eq:Interp_hypoconvex} } \\
        \quad & x_{i+1}=x_{i}-\tfrac{h_{i}}{L} g_{i} , \,\, i \in \overline{0,N-1} \\
        \quad & f_{0}-f_{N} \leq \Delta.
    \end{aligned}
\end{align}


\begin{proposition}\label{prop:PEP_equiv_no_opt}
The solutions of the optimization problems \eqref{eq:PEP_general_inf_dim} (with initial condition $f_0 - f_N \leq \Delta$) and \eqref{eq:PEP_finite_dimensional_no_opt} are the same.
\end{proposition}
\begin{proof}
    Problem \eqref{eq:PEP_finite_dimensional_no_opt} is a relaxation of the PEP \eqref{eq:PEP_general_inf_dim}. Therefore it is sufficient to show that for any feasible solution of \eqref{eq:PEP_finite_dimensional_no_opt},  $\bar{\mathcal{T}}=\big\{(\bar{x}_i,\bar{g}_i,\bar{f}_i)\big\}_{i \in \mathcal{I}}$, there exists a smooth hypoconvex function $f \in \mathcal{F}_{\mu,L}$, with $L>0$, $\mu \in \big(-\infty, 0\big]$, such that $f(\bar{x}_i)=f_i$, $\nabla (\bar{x}_i)=g_i$, $\forall i \in \mathcal{I}$. Since $\bar{\mathcal{T}}$ is a feasible point of problem \eqref{eq:PEP_finite_dimensional_no_opt}, the assumptions of Theorem \ref{thm:interp_hypo_characterization_min} are valid, hence there exists such a function.
\end{proof}

\paragraph{Initial condition $f_0 - f_* \leq \Delta$.} In this case, the discretized PEP is
\begin{align}\label{eq:PEP_finite_dimensional}
    \hspace{-0.1cm}
    \begin{aligned}
        \maximize_{\{(x_i,g_i,f_i)\}_{i \in \mathcal{I^*}}} \quad & \min _{0 \leq i \leq N} \big\{\left\|g_{i}\right\|^2\big\} \\
        \stt \quad & \big\{(x_i,g_i,f_i)\big\}_{i \in \mathcal{I}^*} \text{  satisfy \eqref{eq:Interp_hypoconvex} } \\
        \quad & x_{i+1}=x_{i}-\tfrac{h_{i}}{L} g_{i}, \,\, i \in \overline{0,N-1} \\
        \quad & f_i- \tfrac{1}{2L}\|g_i\|^2 - f_*  \geq 0, i \in \,\, \overline{0,N} \\
        \quad & f_{0}-f_* \leq \Delta.
    \end{aligned}
\end{align}
Note the extra condition characterizing the global minimum value $f_*$, i.e., $f_i- \tfrac{1}{2L}\|g_i\|^2 - f_* \geq 0$. This descent-lemma type condition is necessary in order to properly upper bound $f_*$ with respect to the iterates. Moreover, from the second part of Theorem \ref{thm:interp_hypo_characterization_min} we have that there exists a function $f$ whose global minimum is attained by performing a gradient step \eqref{eq:GM_it_fixed_steps} with $h=1$ from one of the iterates. 
\begin{proposition}\label{prop:PEP_equiv}
The solutions of the optimization problems \eqref{eq:PEP_general_inf_dim} (with initial condition $f_0-f_* \leq \Delta$) and \eqref{eq:PEP_finite_dimensional} are the same.
\end{proposition}
\begin{proof}
    Problem \eqref{eq:PEP_finite_dimensional} is a relaxation of the PEP \eqref{eq:PEP_general_inf_dim}, therefore it is sufficient to show that for any feasible solution of \eqref{eq:PEP_finite_dimensional}, $\bar{\mathcal{T}}^*=\big\{(\bar{x}_i,\bar{g}_i,\bar{f}_i)\big\}_{i \in \mathcal{I}^*}$, there exists a smooth hypoconvex function $f \in \mathcal{F}_{\mu,L}$, with $L>0$, $\mu \in \big(-\infty, 0\big]$, such that $f(\bar{x}_i)=f_i$, $\nabla (\bar{x}_i)=g_i$, $\forall i \in \mathcal{I}$, and $\min_{x \in \mathbb{R}^d} f(x) \geq f_*$. Since $\bar{\mathcal{T}^*}$ is a feasible point of problem \eqref{eq:PEP_finite_dimensional}, the assumptions of Theorem \ref{thm:interp_hypo_characterization_min} are valid, hence there exists such a function.
\end{proof}

The programs \eqref{eq:PEP_finite_dimensional_no_opt} and \eqref{eq:PEP_finite_dimensional} are intractable; therefore we relax them by following the same steps as in \cite[Sections 3.2-3.3]{taylor_smooth_2017}. We detail the procedure on short. A tractable convex reformulation is obtained by using a Gram matrix to describe the iterates $x_i$ and the gradients $g_i$. Let
\begin{align*}
    P := 
    \left[
    \begin{array}{lllll}
         g_0 & g_1 & \dots & g_N & x_0
    \end{array}
    \right]
\end{align*}
and the symmetric $(N+2) \times (N+2)$ Gram matrix $G=P^T P$. The iterates of the gradient method $x_i$ can be expressed using the gradients and the initial value $x_0$: 
{\normalsize
\begin{align*}
    x_i=x_0 +\tfrac{1}{L}\sum\limits_{k=0}^{i-1} h_k g_k    .
\end{align*}
}%
One can replace the iterations $x_i$ in the interpolation conditions and reformulate all the constraints from the PEPs \eqref{eq:PEP_finite_dimensional_no_opt} and \eqref{eq:PEP_finite_dimensional} in terms of entries of $G$, along with function values $f_i$. Then the following tractable SDPs are obtained:

Convex relaxation of \eqref{eq:PEP_finite_dimensional_no_opt}:
\begin{align}\label{eq:PEP_SDP_no_opt}
    \begin{aligned}
        \maximize_{G,\,l,\,x_0,\,\{(g_i,f_i)\}_{ i \in \mathcal{I} }} \quad & l \\
        \stt \quad & f_i-f_j + \trace (A_{ij} G) \geq 0, \,\, i \neq j  \\
        %
        \quad &  f_N - f_0 - \Delta \geq 0 \\
        \quad &  G_{ii} - l \geq 0, \, i \in \overline{0,N} \\
        \quad &  G \succeq 0
    \end{aligned}
\end{align}%

Convex relaxation of\eqref{eq:PEP_finite_dimensional}
\begin{align}\label{eq:PEP_SDP}
    \begin{aligned}
        \maximize_{G,\,l,\,x_0,\,\{(g_i,f_i)\}_{ i \in \mathcal{I}^* }} \quad & l \\
        \stt \quad & f_i-f_j + \trace (A_{ij} G) \geq 0, \,\, i \neq j  \\
        \quad &  f_i - \frac{1}{2L}G_{ii} - f_* \geq 0, \, i \in \overline{0,N} \\
        \quad &  f_* - f_0 - \Delta \geq 0 \\
        \quad &  G_{ii} - l \geq 0, i \in \overline{0,N} \\
        \quad &  G \succeq 0
    \end{aligned}
\end{align}%

The $A_{ij}$ matrices are formed according to the interpolation conditions. A detailed explanation of computing $A_{ij}$ and obtaining the SDP formulations is given in \cite[Section 3.3]{taylor_smooth_2017}. The relaxation is exact for $d \geq N+2$ (the large-scale setting), while for smaller dimensions one can introduce a nonconvex rank constraint; for more details, see \cite[Theorem 5]{taylor_smooth_2017}.


\paragraph{\textbf{Software tools}} One can solve \eqref{eq:PEP_SDP_no_opt} and \eqref{eq:PEP_SDP} using an SDP solver. Alternatively, problems \eqref{eq:PEP_finite_dimensional_no_opt} and \eqref{eq:PEP_finite_dimensional} can be directly solved using the Matlab toolbox PESTO \cite{PESTO} or the Python toolbox PEPit \cite{PESTO_python}.

    	\clearpage
        \section{Proof of Theorem \ref{thm:meta_thm_one_step_p_zero_step_q} on structure of the rates} \label{proof:thm:meta_thm_one_step_p_zero_step_q}


\begin{proof}[Proof of Theorem \ref{thm:meta_thm_one_step_p_zero_step_q}]

Inequality \eqref{eq:one_step_upper_bound} from the assumptions states
\begin{align*}
    \min\{\|\nabla f(x_i)\|^2 \,,\, \|\nabla f(x_{i+1})\|^2\} {}\leq {}
    \frac{f(x_i)-f(x_{i+1})}{p_i}, \quad \forall\, 0 \leq i \leq N-1.
\end{align*}

\paragraph{Proof of \eqref{eq:meta_thm_only_p_big_rate}.} Let $S:= \sum_{i=0}^{N-1} p_i$. By multiplying each inequality from above with $\frac{p_i}{S}$ and summing everything we get:
\begin{align*}
    \sum_{i=0}^{N-1} \frac{p_i}{S}\min\{\|\nabla f(x_i)\|^2 \,,\, \|\nabla f(x_{i+1})\|^2\} 
    {}\leq {}
    \sum_{i=0}^{N-1} \frac{p_i}{S}\frac{f(x_i)-f(x_{i+1})}{p_i}
    {} = {}
    \frac{f(x_0)-f(x_{N})}{S}.
\end{align*}
Moreover, the left-hand side is lower bounded by the minimum gradient norm, since
\begin{align*}
\min_{0 \leq i \leq N}\{\|\nabla f(x_i)\|^2\}
{}={} 
\min_{0 \leq i \leq N}\{\|\nabla f(x_i)\|^2\}   \sum_{i=0}^{N-1} \frac{p_i}{S}
{}\leq{}
\sum_{i=0}^{N-1} \frac{p_i}{S}\min\{\|\nabla f(x_i)\|^2 \,,\, \|\nabla f(x_{i+1})\|^2\} 
\end{align*}
(we used the fact that coefficients $\frac{p_i}{S}$ are positive and sum to one).
Therefore, we obtain \eqref{eq:meta_thm_only_p_big_rate}
\begin{align*}
     \min_{0 \leq i \leq N}\{\|\nabla f(x_i)\|^2\}
     {}\leq {}
     \frac{f(x_0)-f(x_{N})}{S}.
\end{align*}

\paragraph{Proof of \eqref{eq:meta_thm_p_and_q_big_rate}.} Let $S_*:= q + \sum_{i=0}^{N-1} p_i$. Multiplying each inequality \eqref{eq:one_step_upper_bound} with $\frac{p_i}{S_*}$ and summing everything we get:
\begin{align*}
    \sum_{i=0}^{N-1} \frac{p_i}{S_*}\min\{\|\nabla f(x_i)\|^2 \,,\, \|\nabla f(x_{i+1})\|^2\} 
    {}\leq {}
    \frac{f(x_0)-f(x_{N})}{S_*}.
\end{align*}
We write inequality \eqref{eq:zero_step_upper_bound} for $x=x_N$  , 
\begin{align*}
    q \|\nabla f(x_N)\|^2 \leq f(x_N) - f_*,
\end{align*}
and multiply by $\frac{1}{S_*}$, then add it to the above inequality to obtain
\begin{align*}
    \frac{q}{S_*} \|\nabla f(x_N)\|^2 +
    \sum_{i=0}^{N-1} \frac{p_i}{S_*}\min\{\|\nabla f(x_i)\|^2 \,,\, \|\nabla f(x_{i+1})\|^2\} 
    {}\leq {}
    \frac{f(x_0)-f_*}{S_*}.
\end{align*}
Since coefficients $\frac{p_i}{S^*}$ and $\frac{q}{S^*}$ are positive and sum to one, the left-hand side is again lower bounded using the minimum gradient norm:
\begin{align*}
    \min_{0 \leq i \leq N}\{\|\nabla f(x_i)\|^2\}
    {}\leq {}
    \frac{q}{S_*} \|\nabla f(x_N)\|^2 +
    \sum_{i=0}^{N-1} \frac{p_i}{S_*}\min\{\|\nabla f(x_i)\|^2 \,,\, \|\nabla f(x_{i+1})\|^2\}.
\end{align*}
Hence we obtain \eqref{eq:meta_thm_p_and_q_big_rate}
\begin{align*}
     \min_{0 \leq i \leq N}\{\|\nabla f(x_i)\|^2\}
     {}\leq {}
     \frac{f(x_0)-f_*}{S_*}.
\end{align*}

\end{proof}
    	
        \clearpage
        \section{Proof of the convergence rate}\label{appendix:proof:wc_rate_hypo}

This section proves the central result of the paper, i.e., the convergence rate for smooth hypoconvex functions -- \textbf{Theorem \ref{thm:wc_GM_hypo}}.
\begin{proof}[Proof of Theorem \ref{thm:wc_GM_hypo}]
Using the results of Theorem \ref{thm:meta_thm_one_step_p_zero_step_q}, we see that it is enough to examine any pair of consecutive iterates $x_i$ and $x_{i+1}$, which are linked with $x_{i+1}=x_i - \frac{h_i}{L} g_i$  (recall we define $g_i=\nabla f(x_i)$), and prove that they satisfy inequality \eqref{eq:one_step_upper_bound} for some constant $p_i$.

Moreover, since our analysis will be iteration-independent (i.e.,\@ it will not depend on previous iterations), we will show that the constant $p_i$ only depends on the step size of the current iteration, which means that $p_i = P(h_i)$ for some function $P(\cdot)$. 

Consider without loss of generality the first iteration ($i=0$), and denote $\frac{h}{L}$ the step size (we use $h$ instead of $h_0$ to simplify notations). We want to show that 
\begin{align}\label{eq:one_step_upper_bound_i=0}
        \min\{\|\nabla f(x_0)\|^2 \,,\, \|\nabla f(x_{1})\|^2\} \leq 
        \frac{f(x_0)-f(x_{1})}{P(h)}
\end{align}
holds for some function $P(\cdot)$. 
It is well-known that function values decrease strictly at each step of the gradient method when $h_i \in (0,2)$, therefore function $P(\cdot)$ in \eqref{eq:one_step_upper_bound_i=0} will take positive values. 
However this decrease is often proved for functions in $\mathcal{F}_{-L, L}(\mathbb{R}^d)$, i.e., smooth nonconvex functions (see e.g., \cite[Equation (1.2.19)]{Nesterov_cvx_lectures}), which is more restrictive than our class of hypoconvex functions $\mathcal{F}_{\mu, L}(\mathbb{R}^d)$ when $\mu<-L$. Hence, for completeness, we provide a proof below that only requires a bound on the upper curvature $L$.

To obtain the second type of rate involving $f(x_N)-f_*$, we only need to add a proof of inequality \eqref{eq:zero_step_upper_bound} in Theorem \ref{thm:meta_thm_one_step_p_zero_step_q}. It turns out that constant $q=\frac{1}{2L}$  is valid for our class of hypoconvex functions $\mathcal{F}_{\mu, L}(\mathbb{R}^d)$, which again for completeness we prove below.


Finally, combining inequalities \eqref{eq:one_step_upper_bound} involving $P(\cdot)$ and \eqref{eq:zero_step_upper_bound} involving $q$ in Theorem \ref{thm:interp_hypo_characterization_min} results in the claimed convergence rates \eqref{eq:GM_hypo_rate_just_p} and \eqref{eq:GM_hypo_rate_both_p_and_q} in Theorem \ref{thm:wc_GM_hypo}.

\paragraph{Proof that function values decrease.} 
From the upper curvature assumption we know that function 
$\varphi(x) = \frac{L}{2} \|x\|^2 - f(x)$ is convex (see Definition \ref{def:upper_lower_curvature}).
Using the first-order characterization of convex functions implies that 
\begin{align*}
    \varphi(x_1)
    &\ge 
    \varphi(x_0) + \langle \nabla \varphi(x_0) \,,\, x_0-x_1 \rangle \Leftrightarrow \\
    \frac{L}{2} \|x_1\|^2 - f(x_1)
    &\ge 
    \frac{L}{2} \|x_0\|^2- f(x_0) + \langle L x_0 - \nabla f(x_0) \,,\, x_1-x_0 \rangle
\end{align*}
which after algebraic simplifications is equivalent to 
$$f(x_1) \le f(x_0) + \langle \nabla f(x_0) \,,\, x_1-x_0 \rangle + \frac{L}{2}\| x_1 - x_0 \|^2$$
(often called the descent lemma, or quadratic upper bound). Using $x_1=x_0-\frac{h}{L} g_0$ we find that 
\[ f(x_1) \le f(x_0) - \frac{h}{L} \langle \nabla f(x_0) \,,\, \nabla f(x_0) \rangle + \frac{L}{2} \left\| \frac{h}{L} \nabla f(x_0) \right\|^2 \] 
and, simplifying again, we obtain 
\[ f(x_0)-f(x_1) \ge \frac{h(2-h)}{2L} \| \nabla f(x_0) \|^2 \] 
showing that function values decrease when $h \in (0,2)$ for all functions in $\mathcal{F}_{\mu, L}(\mathbb{R}^d)$.

\paragraph{Proof of inequality \eqref{eq:zero_step_upper_bound} involving $q$.} We want to show that inequality \eqref{eq:zero_step_upper_bound} holds for $q=\frac{1}{2L}$, which says 
\[ f(x) - f_* \geq \frac{1}{2L} \|\nabla f(x)\|^2 , \forall x \in \mathbb{R}^d. \]
Since the decrease property holds for any $x_0$ and $h \in (0,2)$, we can use it for $x_0=x$ and $h=1$, and combine with $f(x_1) \ge f_*$ to obtain \[ f(x) - f_* \ge f(x) - f(x_1) \ge \frac{1}{2L} \| \nabla f(x) \|^2, \] showing the desired inequality for all functions in $\mathcal{F}_{\mu, L}(\mathbb{R}^d)$.

\paragraph{Proof of inequality  \eqref{eq:one_step_upper_bound} involving $P(\cdot)$.} 

From the interpolation condition \eqref{eq:Interp_hypoconvex} which holds for any hypoconvex function in $\mathcal{F}_{\mu, L}(\mathbb{R}^d)$ we have
\begin{align*}
\begin{aligned}
    f_{0}-f_{1}-\big\langle g_{1}, x_{0}-x_{1}\big \rangle 
    \geq 
    \frac{1}{2\big(1-\kappa \big)}\Big(\frac{1}{L}\big\|g_{0}-g_{1}\big\|^{2} + 
    \kappa L \big\|x_{0}-x_{1}\big\|^2-
    2 \kappa \big\langle g_{1}-g_{0}, x_{1}-x_{0}\big\rangle\Big) \\
    f_{1}-f_{0}-\big\langle g_{0}, x_{1}-x_{0}\big \rangle 
    \geq 
    \frac{1}{2\big(1-\kappa\big)}\Big(\frac{1}{L}\big\|g_{1}-g_{0}\big\|^{2} + 
    \kappa L\big\|x_{1}-x_{0}\big\|^2-
    2 \kappa \big\langle g_{0}-g_{1}, x_{0}-x_{1}\big\rangle\Big)
\end{aligned}    
\end{align*}
(recall that we defined $\kappa=\frac{\mu}{L}$).
Using $x_1 = x_0 - \tfrac{h}{L} g_0$ and grouping by inner products, the two inequalities can be rewritten as
\begin{align*} 
    f_{0}-f_{1}
    &\geq 
    \frac{1}{2L\big(1-\kappa \big)}\big\|g_{0}-g_{1}\big\|^{2} + 
    \frac{\kappa h^2 - 2\kappa h}{2L\big(1-\kappa \big)} \big\|g_{0}\big\|^2 +
    \frac{2 h}{2L\big(1-\kappa \big)} \big\langle g_{1}, g_{0}\big\rangle \\
     f_{1}-f_{0}
    &\geq 
    \frac{1}{2L\big(1-\kappa \big)}\big\|g_{0}-g_{1}\big\|^{2} + 
    \frac{\kappa h^2 - 2h}{2L\big(1-\kappa \big)} \big\|g_{0}\big\|^2 +
    \frac{2 \kappa h}{2L\big(1-\kappa \big)} \big\langle g_{1}, g_{0}\big\rangle.
\end{align*}
By rewriting the inner product in terms of squares, i.e.,
\begin{align*}
    2 \big\langle g_{1}, g_{0}\big\rangle = \|g_0\|^2 + \|g_1\|^2 - \|g_0 - g_1\|^2,
\end{align*}
we get
%
\begin{align} \label{eq:hypo_interp_(0,1)_2}
    f_{0}-f_{1}
    &\geq 
    \frac{1-h}{2L\big(1-\kappa \big)}\big\|g_{0}-g_{1}\big\|^{2} + 
    \frac{\kappa h^2 - 2\kappa h + h}{2L\big(1-\kappa \big)} \big\|g_{0}\big\|^2 +
    \frac{h}{2L\big(1-\kappa \big)} \|g_1\|^2 \\
     f_{1}-f_{0}
    &\geq 
    \frac{1-\kappa h}{2L\big(1-\kappa \big)}\big\|g_{0}-g_{1}\big\|^{2} + 
    \frac{\kappa h^2 - 2h + \kappa h}{2L\big(1-\kappa \big)} \big\|g_{0}\big\|^2 +
    \frac{\kappa h}{2L\big(1-\kappa \big)} \|g_1\|^2. \label{eq:hypo_interp_(1,0)_2}
\end{align}

We divide the analysis in two cases: (i) $h \in (0,1]$ and (ii) $h\in [1, \bar{h}(\kappa)]$.

\paragraph{Case (i): $h \in (0,1]$.}
From \eqref{eq:hypo_interp_(0,1)_2} we obtain
\begin{align*}
    \frac{\kappa h^2 - 2\kappa h + h}{2L\big(1-\kappa \big)} \big\|g_{0}\big\|^2 +
    \frac{h}{2L\big(1-\kappa \big)} \|g_1\|^2
    {} \leq {}
    f_{0}-f_{1} - 
    \frac{1-h}{2L\big(1-\kappa \big)}\big\|g_{0}-g_{1}\big\|^{2} 
    {} \leq {}
    f_{0}-f_{1},
\end{align*}
since $h\in(0,1]$. The left-hand side can be lower bounded using $\min \{\|g_0\|^2 \,,\, \|g_1\|^2\}$:
\begin{align*}
    \frac{\kappa h^2 - 2\kappa h + 2h}{2L\big(1-\kappa \big)}  \min \{\|g_0\|^2 \,,\, \|g_1\|^2\}
    {} \leq {}
    f_{0}-f_{1},
\end{align*}
hence we have that $P(h)$ is given for all $h \in (0,1]$ by
\begin{align}\label{eq:proof:p_0_h_ll_1}
    P(h) = \frac{1}{2L} \frac{h(\kappa h - 2\kappa + 2)}{1-\kappa} = \frac{1}{2L} \big[ 2h - h^2 \frac{-\kappa}{1-\kappa} \big ].
\end{align}

\paragraph{Case (ii): $h \in [1,\bar{h}(\kappa)] \subseteq [1,2)$.}
We multiply \eqref{eq:hypo_interp_(1,0)_2} by a nonnegative constant $\beta$ and add it to \eqref{eq:hypo_interp_(0,1)_2}:
\begin{align*}
\begin{aligned}
 \big(f_{0}-f_{1}) {}+{} & \beta \big(f_{1}-f_{0}\big)
    {}\geq{}  \\
    \Bigg[
    &\frac{1-h}{2L\big(1-\kappa \big)}\big\|g_{0}-g_{1}\big\|^{2} + 
    \frac{\kappa h^2 - 2\kappa h + h}{2L\big(1-\kappa \big)} \big\|g_{0}\big\|^2 +
    \frac{h}{2L\big(1-\kappa \big)} \|g_1\|^2
    \Bigg] 
    + \\
    \beta
    \Bigg[
    &\frac{1-\kappa h}{2L\big(1-\kappa \big)}\big\|g_{0}-g_{1}\big\|^{2} + 
    \frac{\kappa h^2 - 2h + \kappa h}{2L\big(1-\kappa \big)} \big\|g_{0}\big\|^2 +
    \frac{\kappa h}{2L\big(1-\kappa \big)} \|g_1\|^2
    \Bigg].
\end{aligned}
\end{align*}
We group the terms and obtain:
\begin{align*}
\begin{aligned}
\big(1-\beta\big) \big(f_{0}-f_{1}\big)
    & {}\geq{}
    \frac{1}{2L\big(1-\kappa \big)}\Big[(1-h) + \beta (1-\kappa h) \Big] \big\|g_{0}-g_{1}\big\|^{2} \\
    & + \frac{1}{2L(1-\kappa)} \Big[(\kappa h^2 - 2\kappa h + h) + \beta (\kappa h^2 - 2h + \kappa h) \Big]  \big\|g_{0}\big\|^2  \\
    &  + 
    \frac{h}{2L\big(1-\kappa \big)} \big[ 1 + \kappa \beta \big] \|g_1\|^2 .
\end{aligned}    
\end{align*}
Observe that norm of the difference of gradients in the first term of the right-hand side vanishes by setting 
\begin{align*}
    \beta = \frac{h-1}{1-\kappa h}.
\end{align*}



With this choice of $\beta$ (note that $\beta \ge 0$ since $\kappa \le 0$ and $h \ge 1$) we find
\begin{align*}
\begin{aligned}
\frac{2 - h(1+\kappa)}{1-\kappa h} \big(f_{0}-f_{1}\big)
    & {}\geq{}
    \frac{1}{2L(1-\kappa)}
    \frac{h(1-\kappa)[\kappa h^2 - 2h(1+\kappa) + 3]}{1-\kappa h}
     \big\|g_{0}\big\|^2  \\
    &  + 
    \frac{h}{2L\big(1-\kappa \big)} \frac{1-\kappa}{1-\kappa h} \|g_1\|^2 .
\end{aligned}    
\end{align*}
Canceling the terms we get
\begin{align*}
\begin{aligned}
\big[2 - h(1+\kappa)\big]\big(f_{0}-f_{1}\big)
    {}\geq{}
    \frac{h}{2L}
    \big[\kappa h^2 - 2h(1+\kappa) + 3\big]
     \big\|g_{0}\big\|^2  + 
    \frac{h}{2L} \|g_1\|^2 
\end{aligned}    
\end{align*}
or, equivalently,
\begin{align}\label{eq:proof_final_h_gg_1}
\begin{aligned}
    \frac{h\big[\kappa h^2 - 2h(1+\kappa) + 3\big]}{2L \big[2 - h(1+\kappa)\big]}
     \big\|g_{0}\big\|^2  + 
    \frac{h}{2L \big[2 - h(1+\kappa)\big]} \|g_1\|^2
        {}\leq{}
    f_{0}-f_{1}
\end{aligned}    
\end{align}
(note that $\big[2 - h(1+\kappa)\big] = 2-h - \kappa h$ is nonnegative since $0<h<2$ and $\kappa \le 0$). The left-hand side can be now lower bounded using $\min \{\|g_0\|^2 \,,\, \|g_1\|^2\}$
\begin{align*}
\begin{aligned}
    \frac{h\big[\kappa h^2 - 2h(1+\kappa) + 4\big]}{2L \big[2 - h(1+\kappa)\big]}
    \min \{\|g_0\|^2 \,,\, \|g_1\|^2\} 
        {}\leq{}
    f_{0}-f_{1},
\end{aligned}    
\end{align*}
where the leading coefficient is the sum of the two coefficients in front of $\|g_0\|^2$ and $\|g_1\|^2$. For this last step to be valid we need each of those two coefficient to be  nonnegative, which requires that 
\[ \kappa h^2 - 2h(1+\kappa) + 3 \ge 0 \]
hence limits how large step size $h$ can be. This condition is equivalent to 
$h  {}\leq{} \bar{h}(\kappa)= \tfrac{3}{1+\kappa+\sqrt{1-\kappa+\kappa^2}}$. In particular, $\bar{h}(\kappa \rightarrow -\infty) = 2$ (i.e., full domain), $\bar{h}(\kappa = -1) = \sqrt{3}$ (i.e., the limit from \cite[Theorem 2]{abbaszadehpeivasti2021GM_smooth}), and $\bar{h}(\kappa = 0) = \tfrac{3}{2}$ (i.e., the coverage for the convex case; see also Proposition \ref{prop:cvx_rate}).

Hence we have proved for any $h \in [1,\bar{h}(\kappa)]$ that
\begin{align}\label{eq:proof:p_0_h_gg_1}
    P(h) = \frac{1}{2L} \frac{h\big[\kappa h^2 - 2h(1+\kappa) + 4\big]}{2 - h(1+\kappa)} = \frac{1}{2L} \frac{h(2-h)(2-\kappa h)}{2 - h(1+\kappa)}.
\end{align}

The bound of Theorem \ref{thm:wc_GM_hypo} is then easily obtained: quantity $p(h,\kappa)$ in \eqref{eq:pi_hi_rate_GM}  comes from \eqref{eq:proof:p_0_h_ll_1} and \eqref{eq:proof:p_0_h_gg_1} after multiplication by $2L$.

\paragraph{Comment.} For our particular setup, i.e., gradient method applied on $f \in \mathcal{F}_{\mu,L}$, quantity $P(h)$ is the solution of the PEP \eqref{eq:PEP_finite_dimensional_no_opt} for one iteration. Solving it gave us strong hints about the rate and served as inspiration for the proof. However, the proof is derived independently and holds without relying on the PEP.

\end{proof}
        
        \clearpage
        \section{Proofs of tightness}\label{appendix:tightness_proofs_wc_funcs}

\subsection{Tightness for short steps}

\begin{proof}[Proof of Proposition \ref{prop:tightness_low_gamma}]\label{proof:prop:tightness_low_gamma}
Given the upper bounds from \eqref{eq:GM_hypo_rate_just_p} and \eqref{eq:GM_hypo_rate_both_p_and_q}, respectively, for the regime with $h_i \leq 1$, we construct for each a worst-case function $f \in \mathcal{F}_{\mu,L}$ for which these bounds are reached, therefore demonstrate their tightness.


\paragraph{Function example for \eqref{eq:GM_hypo_rate_just_p}.}
Let $\Delta:=f(x_0)-f(x_N)$ and $U$ be the square root of the upper bound, i.e., the minimum gradient norm,
{\normalsize
\begin{align*}
    U := \sqrt{ \frac{2L\ \Delta}{ 
        \sum\limits_{i=0}^{N-1} h_i \big( 2  - h_i \frac{-\kappa}{
        1-\kappa} \big)
        } }.
\end{align*}
}%
For all $i \in \{0,\dots,N\}$, let
{\normalsize
\begin{align}\label{eq:wc_func_iterates_no_opt}
\begin{aligned}
    x_i   {}&={} \frac{U}{L} \sum\limits_{j=i}^{N-1} h_j \\
    g_i   {}&={} U \\
    f_i   {}&={} \Delta - \frac{U^2}{2L} \sum\limits_{j=0}^{i-1} h_j \Big(2 - h_j \frac{-\kappa}{1-\kappa} \Big) 
\end{aligned}
\end{align}
}%
and define the points
{\normalsize
\begin{align*}
    \bar{x}_i := x_i - \frac{-\kappa}{1-\kappa} \frac{h_i}{L} U \in \big[x_{i+1},x_{i}\big], \quad i \in \big\{0,\dots,N-1\big\}.
\end{align*}
}%
Then a worst-case function is the following piecewise quadratic function $f:\mathbb{R}\rightarrow \mathbb{R}$, where $i \in \{0,\dots,N-1\}$:
{\normalsize
\begin{align*}
    \hspace{-0.05in}
    f(x){}={} 
    \left\{
    \arraycolsep=10pt
    \def\arraystretch{1.8}
    \begin{array}{ll}
        \frac{L}{2} (x+\frac{U}{L})^2 - \frac{U^2}{2L}
        & 
        \text{ if }x \in \left(-\infty, x_N\right] \\
        \frac{\mu}{2} \left(x - x_{i+1} \right)^2 + U \left(x - x_{i+1} \right) + f_{i+1}
        &
        \text{ if }x \in [x_{i+1}, \bar{x}_i]
        \\
        \frac{L}{2} \left(x - x_{i} \right)^2 + U \left(x - x_{i}   \right) + f_{i} & 
        \text{ if }x \in [\bar{x}_i, x_i] \\
        \frac{L}{2} \left(x-x_0\right)^2 + U \left(x - x_{0} \right) + f_{0}
        &
        \text{ if }x \in \left[x_{0}, \infty\right)
    \end{array}
    \right.
\end{align*}
}%
By construction, $f(x_i)=f_i$ and $\nabla f(x_i) = g_i$. The curvature is alternating between $\mu$ and $L$, having the iterates $x_i$ and the points $\bar{x}_i$ as inflection points. By construction, $x_N=0$ and $f_N=0$. One can directly check that $f$ satisfies the interpolation conditions \eqref{eq:Interp_hypoconvex} from Theorem \ref{thm:interp_hypo_characterization_min}. Therefore, the bound from $\eqref{eq:GM_hypo_rate_just_p}$ for $h_i \leq 1$ is exact.

\paragraph{Function example for \eqref{eq:GM_hypo_rate_both_p_and_q}.}
Similarly, let $\Delta:=f(x_0)-f_*$ and $U_*$ be the square root of the upper bound, i.e., the minimum gradient norm,
{\normalsize
\begin{align*}
    U_* = \sqrt{ \frac{2L\ \Delta}{1 + 
        \sum\limits_{i=0}^{N-1} h_i \big(2  - h_i \frac{-\kappa}{
        1-\kappa} \big)
        } }.
\end{align*}
}%
For all $i \in \{0,\dots,N\}$, let
{\normalsize
\begin{align}\label{eq:wc_func_iterates}
\begin{aligned}
    x_i   {}&={} \frac{U_*}{L} + \frac{U_*}{L} \sum\limits_{j=\textbf{\textit{i}}}^{N-1} h_j \\
    g_i   {}&={} U_* \\
    f_i   {}&={} \Delta - \frac{U_*^2}{2L} \sum\limits_{j=0}^{i-1} h_j \, \Big(2 - h_j \frac{-\kappa}{1-\kappa} \Big)
\end{aligned}
\end{align}
}%
and define the points
{\normalsize
\begin{align*}
    \bar{x}_i := x_i - \frac{-\kappa}{1-\kappa} \frac{h_i}{L} U_* \in \big[x_{i+1},x_{i}\big], \quad i \in \big\{0,\dots,N-1\big\}.
\end{align*}
}%
Then a worst-case function is the following piecewise quadratic function $f:\mathbb{R}\rightarrow \mathbb{R}$, where $i \in \{0,\dots,N-1\}$:
{\normalsize
\begin{align*}
    \hspace{-0.05in}
    f(x){}={} 
    \left\{
    \arraycolsep=3pt
    \def\arraystretch{1.4}
    \begin{array}{ll}
        \frac{L}{2} x^2
        & 
        x \in \left(-\infty, x_{N}\right] \\
        \frac{\mu}{2} \left(x - x_{i+1} \right)^2 + U_* \left(x - x_{i+1} \right) + f_{i+1}
        &
        x \in [x_{i+1}, \bar{x}_i]
        \\
        \frac{L}{2} \left(x - x_{i} \right)^2 + U_* \left(x - x_{i}   \right) + f_{i} & 
        x \in [\bar{x}_i, x_i] \\
        \frac{L}{2} \left(x-x_0\right)^2 + U_* \left(x - x_{0} \right) + f_{0}
        &
        x \in \left[x_{0}, \infty\right)
    \end{array}
    \right.
\end{align*}
}%
By construction, $f(x_i)=f_i$ and $\nabla f(x_i) = g_i$. The curvature is alternating between $\mu$ and $L$, having the iterates $x_i$ and the points $\bar{x}_i$ as inflection points. The optimal solution is $\big(x_*,f_*\big)=\big(0,0\big)$. One can directly check that $f$ satisfies the interpolation conditions \eqref{eq:Interp_hypoconvex} from Theorem \ref{thm:interp_hypo_characterization_min}, hence the bound from $\eqref{eq:GM_hypo_rate_both_p_and_q}$ for $h_i \leq 1$ is exact.

An illustration of the one-dimensional piecewise quadratic function $f$ is given in Figure \ref{fig:func_example}.

\end{proof}


\subsection{Tightness for convex functions}

Proof of \textbf{Proposition \ref{prop:cvx_rate}} on the tight convergence rates for convex functions when applying a gradient step \eqref{eq:GM_it_fixed_steps} with $h_i \leq \tfrac{3}{2}$:

\begin{proof}[Proof of Proposition \ref{prop:cvx_rate}]\label{proof:prop:cvx_rate}
The \textit{upper bound} results by taking the limit $\kappa = 0$ in Theorem \ref{thm:wc_GM_hypo} and then using the non-increasing property of the gradient norm for convex functions, proved in Appendix \ref{proof:decreasing_gradient}. Because $\kappa := \tfrac{\mu}{L} = 0$, the two upper bounds merge and the tightness of the result covers the range $h\in (0, \tfrac{3}{2}]$. To demonstrate this, worst-case smooth convex functions examples are obtained from the tightness proof 
\ref{proof:prop:tightness_low_gamma} by setting $\kappa=0$.

For the rate with initial condition $f_0-f_N$, an example is
{\normalsize
\begin{align*}
    \hspace{-0.05in}
    f(x){}={} 
    \left\{
    \arraycolsep=10pt
    \def\arraystretch{1.5}
    \begin{array}{ll}
        \frac{L}{2} (x+\frac{U}{L})^2 - \frac{U^2}{2L}
        & 
        \text{ if }x \in \left(-\infty, x_N\right] \\
        U \left(x - x_{i+1}   \right) + f_{i+1} & 
        \text{ if }x \in [x_{i+1}, x_i] \\
        \frac{L}{2} \left(x-x_0\right)^2 + U \left(x - x_{0} \right) + f_{0}
        &
        \text{ if }x \in \left[x_{0}, \infty\right)
    \end{array}
    \right.
\end{align*}
}%
where $U:= \sqrt{\frac{L\big[f(x_0)-f(x_N)\big]}{ 
        \sum\limits_{i=0}^{N-1} h_i}}$ and $x_i$ and $f_i$ are computed by setting $\kappa=0$ in \eqref{eq:wc_func_iterates_no_opt}.

For the rate with initial condition $f_0-f_*$, an example is
{\normalsize
\begin{align*}
    f(x) = 
    \left\{
    \arraycolsep=3pt
    \def\arraystretch{1.5}
    \begin{array}{ll}
        \frac{L}{2} x^2
        & 
        \ x \in \left(-\infty, x_{N}\right] \\
        U \left(x - x_{i+1} \right) + f_{i+1}
        &
        \ x \in [x_{i+1}, {x}_i] \\
        \frac{L}{2} \left(x-x_0\right)^2 + U \left(x - x_{0} \right) + f_{0}
        &
        \ x \in \left[x_{0}, \infty\right)
    \end{array}
    \right.
\end{align*}
}%
where $U_*:=\sqrt{\frac{L\big[f(x_0)-f(x_*)\big]}{\tfrac{1}{2} + 
        \sum\limits_{i=0}^{N-1} h_i}}$ and $x_i$ and $f_i$ are computed by setting $\kappa=0$ in \eqref{eq:wc_func_iterates}.

Note that the two functions are linear for $x \in [x_N, x_0]$ and extended with quadratics of curvature $L$ outside of this interval.


\end{proof}

        \clearpage
        \section{Proofs on the optimal step-size}\label{appendix:proof:gamma_opt}

Proof of \textbf{Proposition \ref{prop:gamma_star}} for step sizes below the threshold $\bar{h}(\kappa)$:
\begin{proof}[Proof of Proposition \ref{prop:gamma_star}]\label{proof:prop:gamma_star}
Minimizing the upper bound from Theorem \ref{thm:wc_GM_hypo} is equivalent with maximizing the general term $p(h,\kappa)$ \eqref{eq:pi_hi_rate_GM}:
{\normalsize
\begin{align*}
    \begin{aligned}
    h_* = \arg \max_{0 < h {} \leq {} \bar{h}(\kappa)} 
        p(h, \kappa).
    \end{aligned}
\end{align*}
}%
We split the analysis according to the intervals $h\in\big(0,1\big]$ and $h \in \big[1, \bar{h}(\kappa)$\big]:
{\normalsize
\begin{align*}
    &c_1 := \max_{h\in(0,1]} \,\,\,\, 2 h - h^2 \frac{-\kappa}{1-\kappa} = 2 - \frac{-\kappa}{1-\kappa} \\
    &c_2 := \max_{h\in[1,\bar{h}(\kappa)]}  2 h - h^3 \frac{-\kappa}{2-(1+\kappa)h} \geq 2 - \frac{\kappa}{2-(1+\kappa)} = c_1.
\end{align*}
}%
Hence, $c_2 \geq c_1$ and $h_*$ belongs to $\big[1, \bar{h}(\kappa)\big]$. We denote by $c(h)$ the objective function from 
{\normalsize
\begin{align}\label{eq:proof_optimal_step_size}
    \begin{aligned}
    h_* = \arg \max_{1 {} \leq {} h {} \leq {} \bar{h}(\kappa)} 
        2 h + \frac{\kappa h^3}{2-h\left(1+\kappa\right)}
    \end{aligned}.
\end{align}
}%
One can check that $c(h)$ is strictly concave on $[1,2)$ and its first derivative is
{\normalsize
    \begin{align*}
    \begin{aligned}
    c'(h) = -\kappa\left(1+\kappa\right) h^3 +
    \left[3\kappa + \left(1+\kappa\right)^2\right] h^2
    -4\left(1+\kappa\right) h +
    4
    \end{aligned}.
    \end{align*}
}%
By solving the optimization problem \eqref{eq:proof_optimal_step_size} we get the optimal step-size expression \eqref{eq:opt_stepsize_min_known_bound}.
\end{proof}

Proof of the optimal step size from \textbf{Proposition \ref{prop:Conjectured_optimal_step}}, based on Conjecture \ref{conjecture:3rd_regime} about the third regime:
\begin{proof}[Proof of Proposition \ref{prop:Conjectured_optimal_step}] \label{proof:prop:Conjectured_optimal_step}
Similar to the proof in Proposition \ref{prop:gamma_star}, we have to solve the maximization problem
{\normalsize
\begin{align*}
    \begin{aligned}
    h_* = \arg \max_{1 {} \leq {} h < 2} 
        2 h + \frac{\kappa h^3}{2-h\left(1+\kappa\right)}
    \end{aligned},
\end{align*}
}%
but in this case on the full domain of step sizes. The objective function $c(h)=2 h + \tfrac{\kappa h^3}{2-h\left(1+\kappa\right)}$ is strictly concave on $[1,2)$ and its first derivative is
{\normalsize
    \begin{align*}
    \begin{aligned}
     c'(h) = -\kappa\left(1+\kappa\right) h^3 +
    \left[3\kappa + \left(1+\kappa\right)^2\right] h^2
    -4\left(1+\kappa\right) h +
    4
    \end{aligned}
    \end{align*}
}%
Because $c'(1)=1>0$ and $c'(2)=4\kappa(1-\kappa)< 0$, there exists an unique solution $h_* \in [1,2)$ of $c'(h)=0$.
\end{proof}

    \end{appendix}
    
\end{document}